\documentclass[aps, prab, amsmath,amssymb,
 final,
tightenlines,
 twoside,
 twocolumn,
 nofloats,
nofootinbib,
 superscriptaddress,
showkeys,
showkeywords]{revtex4-2}

\usepackage{amsthm}
\usepackage{amsfonts}

\usepackage{booktabs}
\usepackage{lineno}
\usepackage{cases}
\usepackage{mathtools}
\usepackage{hyperref}
\usepackage{academicons}
\usepackage{orcidlink}
\definecolor{orcidlogocol}{HTML}{A6CE39}
\newcommand\norm[1]{\left\lVert#1\right\rVert}

\begin{document}
\title{Fairness in Multi-Proposer-Multi-Responder Ultimatum Game}

\author{\orcidlink{0000-0002-0062-3377}Hana Krakovsk\'{a}}
\email[Corresponding author:]{hana.krakovska@meduniwien.ac.at}
\affiliation{Institute of the Science of Complex Systems, Center for Medical Data Science, Medical University of Vienna, Spitalgasse 23, Vienna, 1090, Austria}
\affiliation{Complexity Science Hub, Josefstädter Str. 39, Vienna, 1080, Austria}
\author{\orcidlink{0000-0002-4045-9532}Rudolf Hanel} 
\affiliation{Institute of the Science of Complex Systems, Center for Medical Data Science, Medical University of Vienna, Spitalgasse 23, Vienna, 1090, Austria}
\affiliation{Complexity Science Hub, Josefstädter Str. 39, Vienna, 1080, Austria}
\author{\orcidlink{0000-0002-1698-5495}Mark Broom}
\affiliation{Department of Mathematics, City, University of London, Northampton Square, London, EC1V 0HB, United Kingdom}
 
\date{Version \today}

\begin{abstract}
\noindent The Ultimatum Game is conventionally formulated in the context of two players. Nonetheless, real-life scenarios often entail community interactions among numerous individuals. To address this, we introduce an extended version of the Ultimatum Game, called the Multi-Proposer-Multi-Responder Ultimatum Game. In this model, multiple responders and proposers simultaneously interact in a one-shot game, introducing competition both within proposers and within responders. We derive subgame-perfect Nash equilibria for all scenarios and explore how these non-trivial values might provide insight into proposal and rejection behavior experimentally observed in the context of one vs.~one Ultimatum Game scenarios. Additionally, by considering the asymptotic numbers of players, we propose two potential estimates for a "fair" threshold: either 31.8\% or 36.8\% of the pie (share) for the responder. 
\end{abstract}

\maketitle
\section*{Introduction}
The Ultimatum Game (UG) has been one of the paradigmatic games for studying fairness since its introduction in 1982 by G\"{u}th et al.~\cite{guth1982}. This simple one-shot game consists of a reward and two players with asymmetric roles: proposer and responder. The proposer's task is to suggest a split of the reward between the players. The responder can then accept or reject this split. On acceptance, the reward is split accordingly, and on rejection, both players receive nothing. The theoretical prediction states that a rational self-interested proposer will offer the minimum amount they believe the responder will accept. Similarly, a rational self-interested responder will accept any positive offer and remain indifferent between accepting or rejecting a zero offer. Thus, the subgame-perfect Nash equilibrium~\cite{gintis2000}, dictates the proposer to offer the smallest possible unit of the reward 
which the responder then accepts. 

However, in experiments, real-world players do not play as predicted by this theory (see e.g.~Camerer~\cite{camerer2011} and Henrich et al.~\cite{henrich2005}).
In Western, educated, industrialised, rich, democratic (W.E.I.R.D) societies, responders reject offers of less than 20$\%$ of the reward with a probability of one-half and almost always accept proposals of 40$\%$ to 50$\%.$ Proposers' modal offers are usually between $40\%$ to $50\%$ and mean offers between $30\%$ to $40\%$~\cite{camerer2011}. For example, Oosterbeek et al.~\cite{oosterbeek2004} found in their comprehensive review a mean offer of $40\%.$ On the other hand, experimental results on UG-like scenarios (see again~\cite{henrich2005}), played in various small-scale societies, report more variable proposer and responder thresholds, ranging from greedy to generous. 
Interestingly, among the most generous proposers are tribal whaler societies~\cite{alvard2004}, whose lifestyle requires high levels of cooperation and mutual trust. Elsewhere, it is common for low offers to be both proposed and accepted~\cite{patton2004,henrich2004}. 

Additionally, the sensitivity to biases, including anonymity, stake level, and context effects have been studied, and it was concluded they do not seem to shift played thresholds significantly~\cite{larney2019,charness2008,hoffman2000}. 
Moreover, \textit{priming} players, e.g.~by conjuring an imaginary {\rm right} to play proposer, has been found to make players propose more greedily (see Hoffman et al.~\cite{hoffman1994}), although some of the results have been recently contested by Demiral and Mollerstrom~\cite{demiral2020}.
A concise understanding of how human sharing and cooperation behavior precisely depends on context and circumstances (including the complexity of social interactions in terms of multi-player scenarios) is, therefore, still partly to be found.

Whether \textit{fairness} is part of our ancestral behavioural makeup that evolved among apes or can be attributed only to the "modern" self-domesticated human~\cite{wilkins2014,stetka2021} is still debated~\cite{jensen2007,milinski2013,kaiser2012,proctor2013}.

W.E.I.R.D. societies have been shown to follow a unique cognitive trajectory, concerning sharing and cooperation, which can be traced back to a change in inheritance law in 305 AD. It was introduced by the Catholic Church due to its preoccupation with incest, promoting heritage by testament, banning polygamous marriages and marriages to relatives, promoting the newlywed to set up independent households~\cite{henrich2010, schulz2018}. 

This reform, over a period of more than a thousand years, broke up clan structures and thereby fostered individuality over clan-identity, and the interaction and cooperation of individuals across clan-boundaries, which may explain some peculiarities of the W.E.I.R.D. world. In summary, our evolutionary and societal development remains partly veiled in the dusk of the past and can only be accessible through comparative studies (e.g.~\cite{wilkins2014}) and to a \textit{hypothetical realism} that largely remains grounded in mathematical modeling.

Evolutionary game theory is a common framework for studying the emergence of \textit{fairness} (see e.g.~Debove et al.~\cite{debove2015} who review various evolutionary game-theoretic models of the UG type). For instance, some models on networks indicate a dependence of the proposer and responder thresholds on the topology of the social network (see e.g.~Sinatra et al.~\cite{sinatra2009}, Kuperman and Risau-Gusman~\cite{kuperman2008} and Page et al.~\cite{page2000});  others explore the role "social status", may play \cite{harris2020}. Alternatively, Fehr and Schmidt~\cite{fehr1999} propose an inequity-aversion framework in their seminal work. Also, see the review by G\"{u}th and Kocher~\cite{guth2014}.

Regarding generalisations of the UG to multiple players, the majority of theoretical and experimental extensions typically focus on scenarios with either one proposer and many responders or vice versa. For instance, Roth et al.~\cite{roth1991} conducted experiments with proposer competition and one responder, while Fischbacher et al.~\cite{fischbacher2009} explored responder competition. They showed that introducing competition raises the offers in the former and decreases offers in the latter as compared to the one vs.~one UG.  

Additionally, Santos et al.~\cite{santos2015} investigated an evolutionary game-theoretic model with a group decision-making process for responders when faced with an offer from a single proposer. In another paper, Santos and Bloembergen~\cite{santos2019} extend the UG to a group of proposers playing with a group of responders. The group of responders rejects the average proposed offer if it is lower than their average group acceptance threshold. In generosity or envy games there is an additional third dummy player that takes no active role (for details see survey from G\"{u}th and Kocher~\cite{guth2014}). 

In this paper, we introduce an extension of the UG where multiple proposers and multiple responders engage in a one-shot interaction. This extended model seems to be more appropriate in various social contexts, such as the tribal whaler society mentioned above than the classical two-player paradigm. It constitutes a two-sided market, where proposers simultaneously make their offers, and responders, in turn, simultaneously each select an offer from one of the proposers (or choose a number with some probability), or decline all offers.
If a proposer is chosen by at least one responder, they receive the proposed split. Conversely, if no responders select them, they receive nothing. On the other hand, all responders who chose a proposer that was not selected by anyone else will receive their share. In cases where multiple responders choose the same proposer, only one of them, randomly selected, receives the proposed split, while the other responders receive nothing. 

In the context of a simplified labor market, proposers represent potential employers and responders potential employees. We assume that the employers offer identical roles, such as a plumbing job, where the reward is the total revenue, and they must decide how to split it with the employee. If an employer decides to claim too large a share of the revenue, they run the risk of being outbid by other employers and earning nothing. Similarly, if the employees would simply choose the highest offer, they may end up making the same choice as many other responders, resulting in no job opportunity if they are not selected by the probabilistic rule. An essential factor influencing the "optimal" offers of employers is the balance between the number of available jobs (employers) and the number of potential employees. Alternatively, within a biological context, proposers may symbolise plants, offering their energy in the form of nectar, while responders represent the pollinators.
  
A similar concept to the multiplayer UG is sequential bargaining (see e.g.~Rubinstein and Wolinsky~\cite{rubinstein1985}, Li et al.~\cite{li2022}). In these models, the proposer role is passed on around players until everyone agrees with the division. In Li et al.~\cite{li2022} they describe a model with two buyers and two sellers. Sequentially, each participant selects a partner from the opposing group and initiates bargaining by making an offer. If the initial pair reaches an agreement, they quit bargaining and leave with their share. However, in the case of rejection, they remain open to being chosen by another seller or buyer for further bargaining. If not all players agree on a share, in the subsequent round the other group takes the role of the proposer and makes the offers. However, these models differ due to their sequential quality and role switching, whereas in our model we present a one-shot model.
 
This paper is organised as follows. In the following section,
we present a formalisation of the game and 
(for a self-contained description) briefly introduce the subgame-perfect Nash equilibria of the one vs.~one and one vs.~many UG scenarios. 
Then, in \hyperref[sec:2P2R]{Two Proposers and Two Responders}
we provide a detailed analysis of the game involving two proposers and two responders, where we find the subgame-perfect Nash equilibrium of the game. Subsequently, in~\hyperref[sec:generalPR]{$K$ Proposers and $L$ Responders}, 
we derive the solution for the general case of multiple proposers facing multiple responders. 
In the~\hyperref[sec:discussion]{Discussion} we scrutinise the discovered solutions, look closer at their asymptotic properties, and discuss some implications for fairness in human behaviour. Lastly, we summarise our results in the~\hyperref[sec:conclusions]{Conclusions}.
\section*{The Model}
\label{sec:model}
In the following we 
extend the framework of the classical UG to a \textit{Multi-Proposer-Multi-Responder} (MPMR) framework and, at the same time, comment on known results of three special cases: one responder vs.~one proposer, many responders vs.~one proposer, and one responder vs.~many proposers.

The MPMR UG game involves $L \in \mathbb{N}$ responders and $K \in \mathbb{N}$ proposers. Each proposer is endowed with a potential reward of size one that is to be split with one of the responders. The game has two stages. In the first stage, each proposer puts forward a split of the reward denoted as $s_i \in [0,1]$, where ${i \in \{ 1,2,\dots,K\}}$, offering $s_i$ to the responders and ${1-s_i}$ to themselves. In the second stage of the game, responders simultaneously and independently (without the knowledge of other responders' choices) select one of the proposers or select no one. Additionally, they may also use mixed strategies and probabilistically decide between multiple proposers. Responders have full information about all the proposed offers and are impartial towards the proposers themselves. Following the decision of the responders, the payoffs are distributed. If the proposer $i$ was selected by at least one responder, they receive {$1-s_i$}, otherwise, they receive nothing. Similarly, only one responder picked randomly (with probability one divided by the number of individuals that chose the same proposer) among the respective responders receives the offered split $s_i$, the others receive nothing, just as responders who did not choose any proposer would. 

Let us rewrite the payoffs in symbolic terms. Denote $N_{R,i} \in \{0,1,\dots, L\},$ $ i \in \{1,2, \dots, K\}$ as the number of responders that chose the proposer $i$. The payoff $\pi_{P,i}$ of proposer $i$ with offer $s_i \in [0,1]$ is calculated as:
\begin{equation*}
\pi_{P,i}=
\begin{cases}
\begin{aligned}
&1-s_i &\quad \text{if } &N_{R,i}\geq 1\ , \\
&0 &\quad \text{if } &N_{R,i}=0\ . 
\end{aligned}
\end{cases}
\end{equation*}
The payoff $\pi_{R,j}$ of the responder $j, j \in \{1,2,\dots,L\}$ that chose proposer $i$ is given as
\begin{equation*}
\pi_{R,j}=
\begin{cases}
\begin{aligned}
&s_i &\quad   &\text{with probability } \frac{1}{N_{R,i}}\ , \\
&0 &\quad &\text{with probability } \frac{N_{R,i}-1}{N_{R,i}}\ . 
\end{aligned}
\end{cases}
\end{equation*}
The payoff $\pi_{R,j}$ of the responder $j$ who rejected all offers is given as:
\begin{equation*}
\pi_{R,j}=0 \ .
\end{equation*}
 In Fig.~\ref{fig:MRMPUG_graphics} we illustrate the mechanics of the MPMR UG in a simple scenario.
\begin{figure*}
    \centering
    \includegraphics[scale=0.87]{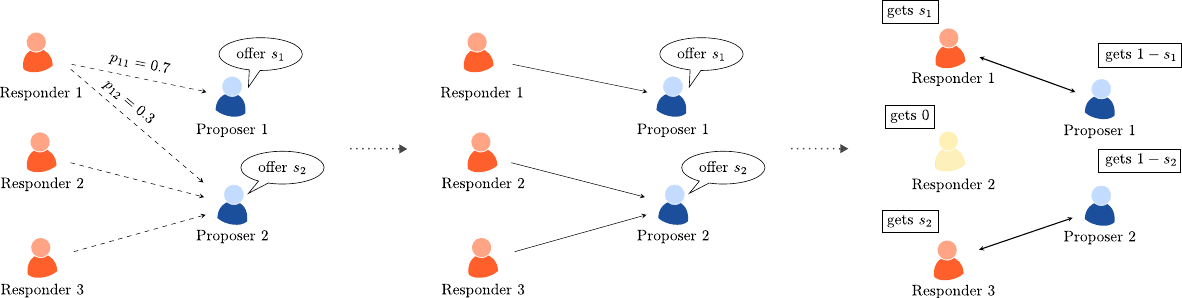} \caption{Graphical representation of the game for the case of $L=3$ responders and $K=2$ proposers. In the first stage, proposers announce their offers, prompting each responder to determine their selection strategy. In this scenario, responder $1$ chooses a mixed strategy while the other two responders play pure strategies. In the second stage, it is probabilistically decided that responder $1$ chooses proposer $1$. Since two responders chose proposer $2,$ another probabilistic realisation determines who gets paired with the proposer. In this case, responder $2$ is unpaired, resulting in a zero payoff. Similarly, if one of the proposers (or both) would not be selected they would receive a zero payoff.}
  \label{fig:MRMPUG_graphics}
\end{figure*}
Before going into further analysis let us first summarise the known results for three basic cases of the game. With one proposer and one responder, the game reduces to a classical UG. If the responder's strategy is to refuse any offer other than $s, s \in [0,1]$ then the proposer offering $s$ corresponds to a Nash equilibrium. Likewise, it constitutes a Nash equilibrium if we presume that once a responder accepts $s$, they will accept any offer $s^*$ greater than $s.$ Consequently, there exists a continuum of Nash equilibria in the UG. To reduce their number, attention can be directed towards the concept of a subgame-perfect Nash equilibrium~\cite{gintis2000}. If the proposer deviates from their strategy and offers slightly less than $s$, in this subgame, the responder's best-reply is to accept the offer. Consequently, no $s > 0$ can be a subgame-perfect equilibrium, since the proposer can always improve their payoff by lowering the offer. Thus, there exists a unique subgame-perfect equilibrium where the proposer offers ${s = 0}$ and the responder accepts all offers. If there exists a grid of possible offers instead of a continuous interval, the second smallest possible offer $\epsilon > 0$ is also a subgame-perfect equilibrium. This is because the responder maximises their payoff regardless of accepting or rejecting the offer of $s = 0$ and may choose to reject a zero offer, in which case the proposer's subgame-perfect equilibrium is to offer $\epsilon.$

Similarly, in the case of multiple responders facing a single proposer, in the subgame-perfect equilibrium responders are offered zero and at least one of them accepts the offer. Once again, in the case of the existence of the second smallest offer, this is a subgame-perfect equilibrium too if all responders reject the zero offer. In the case of proposer competition, where multiple proposers face a single responder, in the subgame-perfect equilibrium at least two proposers offer $s=1$ which the responder accepts (note that the other proposers can offer any split, which subsequently leads to multiple equilibria). For details see e.g.~\cite{fehr1999}.

In the next sections, we will see that the introduction of two-sided competition yields intriguing outcomes. Since only one responder gets their share of the reward if more than one of them select the same proposer, it may not always be optimal for responders to straightforwardly select the best proposal, as others might employ the same approach. Instead, a more efficient strategy might involve granting a non--zero probability of going to the second-best proposal and other alternative offers. In this sense, the two-sided competition problem introduces a minority-game-like situation (as in the El Farol Bar problem~\cite{arthur1994}) into the MPMR UG scenario. Consequently, the proposers are motivated to make offers below the offer of one, as there is a possibility that the second-best offer might still be accepted by some responder. On the other hand, there is still the presence of competition between the proposers which prevents them from offering zero. 

In the next section, we start with the scenario involving two proposers and two responders.
\section*{Two Proposers and Two Responders}\label{sec:2P2R}
In this section, we provide an analysis of the MPMR UG with two proposers and two responders. Our goal is to find subgame-perfect Nash equilibria. We will see that in each subgame with at least one positive offer there are either one or two Nash equilibria for the responders, depending on the combination of offers. Only one of them is also an evolutionarily stable strategy (ESS) and when restricting ourselves to that strategy there is a unique subgame-perfect Nash equilibrium where both proposers offer $s=0.5$ and both responders select each of them with probability $0.5.$
\subsection*{Responders' Nash Equilibrium}
Let us start with denoting the strategy of one particular subgame (i.e.~how much is offered to the responders) for the first proposer as $s_1 \in [0,1]$ and for the second proposer as $s_2\in [0,1].$  Without loss of generality we assume $s_1\leq s_2.$ The mixed strategies of the respective responders (denoted $R_1$ and $R_2$) selecting the respective proposers (denoted $P_1$ and $P_2$) will be given as:
\begin{equation*}
    \begin{aligned}
P(R_1 \text{ choosing } P_1)&=p,  \
P(R_1 \text{ choosing } P_2)&=1-p,\\
P(R_2 \text{ choosing } P_1)&=q, \
P(R_2 \text{ choosing } P_2)&=1-q,\\ 
    \end{aligned}
\end{equation*}
where $p,q \in [0,1]$ (note neither of the responders rejects any offer since this simplification does not prevent us from finding the best reply strategies).
Considering these strategies we can calculate the expected payoffs of responders, denoted by $\Pi_{R,1},\Pi_{R,2} $ for the first and second responder respectively:
\begin{equation*}
\label{eq:payoffs_R1R2}
\begin{aligned}
\Pi_{R,1}(p,q)&=s_1p\left(1-q+\frac{q}{2}\right)+s_2(1-p)\left(q+\frac{1-q}{2}\right),  \\
\Pi_{R,2}(p,q)&=\Pi_{R,1}(q,p) \ .
\end{aligned}
\end{equation*} 
Next, we want to find the best response of the first responder given a set of offers $(s_1,s_2)$ and a fixed mixed strategy $(q,1-q)$ of the second responder. 
We start by finding $p$ that maximises $\Pi_{R,1}.$ Notice the payoff  $\Pi_{R,1}$ is linear in $p.$ Thus, the derivative of $\Pi_{R,1}$ with respect to $p$ is constant for all $p$ and given as:
\begin{equation}
\label{eq:derivativeR1}
\frac{\partial \Pi_{R,1}}{\partial p}=s_1 - \frac{s_2 + qs_1 + qs_2}{2}.
\end{equation}
It is easy to find the best response strategy of the first responder, denoted by $p^*:$   
\begin{equation}
p^*=
    \begin{cases}
    \begin{aligned}
      &[0,1] \quad &\text{if } s_1&=s_2=0,\\
       &[0,1] \quad &\text{if } q&=q_{\rm{crit}}\equiv \frac{2s_1-s_2}{s_1+s_2},\\
        &1 \quad &\text{if } q&<q_{\rm{crit}},\\
       &0 \quad &\text{if } q&>q_{\rm{crit}}.\\       
        \end{aligned}
    \end{cases}
    \label{eq:opt_p}
\end{equation}
We would get mirror results for the best response strategy of the second responder since the responder roles are symmetric.
It is evident from~\eqref{eq:opt_p}, that when the higher offer $s_2$ exceeds twice the value of $s_1$, it is advantageous for both responders to disregard the lower offer $s_1.$ Conversely, in instances where the offers are sufficiently similar ($s_2 \leq 2s_1$), a critical threshold $q_{\rm{crit}}$ emerges. If the second responder plays the threshold strategy $q_{\rm{crit}}$, the first responder's strategy becomes irrelevant. If the responder plays a different strategy to $q_{crit},$ the best response is to choose exclusively the other relatively less occupied proposer with respect to the threshold. As a consequence, the Nash equilibrium for responders in the subgame with fixed offers $s_1 \leq s_2$ is:
\begin{equation}
\begin{aligned}
&\textit{Sc. 0:}&\text{if } &s_1=s_2=0  &p_0,q_0 \in [0,1] \text { or reject.}  \\
&\textit{Sc. A:}&\text{if } &s_2<2s_1  &p_A=q_A=\frac{2s_1-s_2}{s_1+s_2}.  \\ 
&\textit{Sc. B:}&\text{if } &s_2<2s_1  &p_B=0, \hspace{0.5em}q_B=1 \text{ or reversed.}\\
&\textit{Sc. C:}&\text{if } &s_2=2s_1\neq 0   &p_{C}=0,  q_{C} \in [0,1]  \text{ or reversed.}\\
&\textit{Sc. D:}&\text{if } &s_2>2s_1  &p_{D}=q_{D}=0. \\
\end{aligned}\label{eq:scenarios}
\end{equation} 
In \textit{Scenario 0} both proposers offer zero and responders receive a zero payoff, no matter which strategy they choose. In \textit{Scenario D} one of the offers is more than double the amount proposed by the other proposer. Thus, both responders opt for the proposer with the higher offer and discard the other. In \textit{Scenario C}, there is a continuum of Nash equilibria, however, playing the strategy of going to the second proposer with offer $s_2$ is a weakly dominant strategy (this is easy to see, since the payoff is given as $\Pi_{R,1}=s_1( 1+  q-\frac{3}{2} p q)$). 
Notice that when the ratio of the offers is less than two ($s_2<2s_1$),  two types of Nash equilibrium strategies for the responders emerge: symmetric strategies (\textit{Scenario A}) and strategies which require "coordination" of the responders (\textit{Scenario B}). 
Let us now look closer at how these strategies perform when played against each other. We denote the offer levels as $s_1=s,$ where $s\in[0,1]$ and $s_2=s+\delta,$ where  $0\leq \delta < s$ and  $0<s_2\leq 1.$ The payoff matrix for responders playing strategies $p_A, p_B$ against $q_A,q_B$ is shown in Table~\ref{table:payoffs_AB}.
  \begin{table}[h]
\begin{center}
\begin{tabular}{cccc} 
\toprule
   & $q_A $&$q_B=0$ & $q_B=1$ \\  
\midrule \\[-0.5ex]  
 $p_A$ & \large $\frac{3s(s + \delta)}{2(2s + \delta)}$& \large$\frac{3s(s + \delta) + 2\delta(\delta-s)}{2(2s + \delta)}$& \large$\frac{3s(s+\delta)+ 2\delta(2\delta +s)}{2(2s + \delta)}$  \\ [1.5ex] 
$p_B=0$ & \large $\frac{3s(s + \delta )}{2(2s + \delta)} $&\large        $\frac{1}{2}(s + \delta )$&  \large$ s+\delta$\\ [1.5ex]  
$p_B=1$ & \large$\frac{3s(s + \delta)}{2(2s + \delta)}$& \large $s$ & \large $ \frac{1}{2} s$\\ [1ex] 
\bottomrule
\end{tabular}
\end{center}
\caption{Payoff matrix for responders playing strategy $p_A,p_B=1$ and $p_B=0$ in \textit{Scenario A} and \textit{Scenario B}.}
\label{table:payoffs_AB}
\end{table}
Choosing the lower offer with probability zero ($p_B=0$) is the best response against both $q_A$ and $q_B=1,$ however when played against the same strategy $q_B=0$ the strategy is inferior to $p_A$ and $p_B=0.$ Thus, contrary to \textit{Scenario C}, there is no weakly dominant strategy. In order to determine which of the two Nash equilibria will be played by the responders, we can turn to the concept of an evolutionarily stable strategy. As we prove in the Supplementary Information (see Theorem 1) the symmetric strategy $p_A$ is uniquely evolutionarily stable, which means that it is (in this way) superior to other strategies. 

We also explored the replicator dynamics with the co-existence of three types of Nash equilibrium strategies: $\gamma$--players who always choose the strategy of \textit{Scenario A}, $\alpha$--players who always choose the highest offer and $\beta$--players who always choose the lowest offer. The analysis, detailed in the Supplementary Information, reveals that in the stable state $\alpha$-- and $\beta$--players' relative abundances on average yield the same strategy as the \textit{Scenario A} strategy, albeit on a population-wide level rather than an individual one. The abundance of $\gamma$--players depends on the initial conditions. Thus, in the next part, we reduce the analysis to responders playing the Nash equilibrium strategy $p_A.$
\subsection*{Proposers' Nash Equilibrium Strategy}
In this section, we will derive the subgame-perfect Nash equilibrium of the proposers. Without loss of generality, we assume that $s_1\leq s_2.$ As demonstrated in the preceding section, in the cases where $2s_1 \leq s_2,$ both responders choose the proposer with the higher offer and the other proposer receives no reward. Since the abandoned proposer can improve their payoff by increasing their offer to slightly more than half of the other offer, this situation cannot be a Nash equilibrium. For the case where $s_1=s_2=0$, there is a subgame-perfect Nash equilibrium. Here both responders receive zero, no matter what their strategy is. If each responder accepts and chooses different proposers, then both proposers receive the maximal payoff of one. However, this requires responders to accept a zero reward and align, i.e.~to match their choices and can be considered degenerate. Any other matching (or rejection) does not lead to
a subgame-perfect Nash equilibrium. Then one of the proposers (or both) has an expected payoff of less than one and by raising their offer from zero to a sufficiently small offer $\epsilon$, they are able to secure a better payoff, $1-\epsilon$. 

In the following, we will only consider offers satisfying ${0 < s_2 \leq 2s_1}$. The expected payoffs of proposers $1$ and $2,$ with offers $s_1$ and $s_2$ respectively (denoted $\Pi_{P,1}$ and $\Pi_{P,2}$ respectively), are:
\begin{equation}
\begin{aligned}
\Pi_{P,1}&=(1-s_1)\left(pq+p(1-q)+q(1-p)\right)\ ,\\
\Pi_{P,2}&=(1-s_2)((1-p)(1-q)+p(1-q)+q(1-p))\ ,
\end{aligned}
\end{equation}
where $(p,1-p)$ and $(q,1-q)$ are the responders' strategies. We assume that the responders choose the strategy $p_A$  (see ~\eqref{eq:scenarios}) according to the Nash equilibrium \textit{Scenario A} (for reasons why we do not focus on \textit{Scenario B} strategy see above). 
Then we can rewrite the expected payoffs as:
\begin{equation}
    \begin{aligned}
        \Pi_{P,1}&=\frac{3 s_2 (2 s_1 - s_2) (1 - s_1)}{(s_1 + s_2)^2}\ ,\\
        \Pi_{P,2}&=\frac{3 s_1 (2s_2 -  s_1) (1 - s_2)}{(s_1 + s_2)^2}\ .
    \end{aligned}
    \label{opt_prop_payoffA}
\end{equation}
To maximise the payoffs we set $\partial \Pi_{P,i}/\partial s_i=0$, with $i=1,2$, identify the best response of the proposers,
yielding for the first and second proposer:
\[s_1^*= f(s_2)=\frac{s_2^2 + 4s_2}{5s_2 + 2}\quad {\rm and }\quad s_2^*= f(s_1)=\frac{s_1^2 + 4s_1}{5s_1 + 2}\ .\] 
The subgame-perfect Nash equilibrium $(s_1^*,s_2^*)$ has to satisfy $s_1^*=f(f(s_1^*))$ leading to a fourth order polynomial that has four roots $\{0,\frac{1}{2},-3+\sqrt{7},-3-\sqrt{7}\}.$ However, the only  physical, non-trivial solution is $\frac{1}{2}$, which leads to the unique non-degenerate Nash equilibrium of symmetric offers $(\frac{1}{2},\frac{1}{2}).$
In that case, the payoffs of the players are the same in both roles:
\begin{equation}
\begin{aligned}
   \Pi_{R,1}=\Pi_{R,2}= \Pi_{P,1}=\Pi_{P,2}=\frac{3}{8}\ .
\end{aligned}
\end{equation}
Finally, we remark on a potentially counter-intuitive observation: by changing the rules of the game and giving the responders a chance to coordinate for maximizing their common payoff in each subgame, their subgame-perfect Nash equilibrium payoffs are lower than in the MPMR UG.

It is easy to see this from the following. Opting for the strategy where one responder selects the first proposer while the other selects the second results in the responders receiving all of the available reward in every subgame, so more than with any other strategy. Then, the payoffs of the proposers are:
\begin{equation}
    \begin{aligned}
        \Pi^*_{P,1}&=1-s_1 \ ,\quad
        \Pi^*_{P,2}&=1-s_2\ .
    \end{aligned}
    \label{opt_prop_payoffB}
\end{equation}
       However, then the subgame-perfect equilibrium is $(0,0)$ and the responders accept, since, under all offer combinations, the proposers receive their offered split and they have no incentive to raise the offers from zero. Thus, this type of coordination leads to worse outcomes for the responders.
\section*{K Proposers and L Responders}
\label{sec:generalPR}
  Now to the general case of  $K \in \mathbb{N}$ proposers and $L \in \mathbb{N}$ responders where $K,L \geq 2.$
\subsection*{Responders' Strategy}
Let us find the Nash equilibria for the responders for subgames determined by offers from $K$ proposers. Let $p_{ij}=P(\text{$R_i$ choosing $P_j$}),$ $ {i \in \{1,2,\dots,L\}}$, with $j \in \{1,2,\dots,K\}$, denote the probability of the responder $i$ choosing proposer $j$, who offers $s_j\in [0,1].$ Naturally, $\sum_{j=1}^K p_{ij}=1$ and $p_{ij}\geq 0$ and
the expected payoff $\Pi_{R,l}$ of the responder $l$ is:
    \[\Pi_{R,l}=\sum_{i=1}^K s_i W_{li} \cdot p_{li} \ ,\]
where \[ W_{li}=\sum_{j=1}^{L} \frac{1}{j} \left[ \sum_{\alpha:\norm{\alpha}_{1}=j,\alpha_l=1} \prod_{k=1,k\neq l}^{L} p_{ki}^{\alpha_k}(1-p_{ki})^{1-\alpha_k}\right] \ ,\]
with $\alpha \in \{0,1\}^{L}$ and $\norm{.}_1$ being the $L^1$ norm. Each term of the sum in $W_{li}$ describes the probability of precisely $j$ responders (including responder $l$) choosing proposer $i,$ weighted by $\frac{1}{j}.$ Also notice that $W_{li}$ is independent of $p_{li},$ meaning that payoff $\Pi_{R,l}$ is a linear function defined by strategy ${p_l= (p_{l1},p_{l2},\dots p_{lK})}$ and scalars $(s_1 W_{l1},s_2 W_{l2},\dots, s_K W_{lK}).$
In order to find the best response of responder $l$ facing the other responders, we have to solve a constrained optimisation problem on a linear function. Without loss of generality, we set the ordering of proposers to be determined by:
\[s_K W_{lK} \geq s_i W_{li} \ ,\] 
for all $i \in \{1,2,\dots,K\}$ and analyse the derivative of the payoff of responder $l$ with respect to ${p_{li}, i\in \{1, \dots, K-1\}}$ where $p_{lK}=1-\sum_{j=1}^{K-1} p_{lj}:$
\begin{equation}\label{eq:resp_derivative}
\begin{aligned}
    \frac{\partial \Pi_{R,l}}{\partial p_{li}} =   s_i W_{li}-s_K W_{lK}\ .
\end{aligned}
\end{equation}
Thanks to the ordering, we know the derivative (\eqref{eq:resp_derivative}) has only non-positive values. The best response strategy is to set as zero all $p_{li}$ for which  $\frac{\partial \Pi_{R,l}}{\partial p_{li}}<0$ and when determining the remaining probabilities where $\frac{\partial \Pi_{R,l}}{\partial p_{li}}=0$ the responder is ambivalent. Translating these best responses to Nash equilibria for each subgame is more intricate in the higher-dimensional scenario than in the simpler two-on-two case.

Once again, there is a \textit{Scenario 0}-like regime where all proposers offer zero and the responders' strategy is inconsequential-- responders receive zero in all cases. There are also regimes similar to \textit{Scenarios C} and \textit{D} where some of the proposers offer too little compared to the other proposers and their offers get discarded by every responder. In regimes where all offers are sufficiently similar to each other, both symmetric and asymmetric Nash equilibria resembling \textit{Scenario A} and \textit{Scenario B} can emerge, where each proposer has a positive probability of being selected. As an example, consider the case of two responders and three proposers offering identical offers $s \in (0,1].$ Let $p=(p_{11},p_{12},p_{13})$ and $q=(p_{21},p_{22},p_{23})$ represent the responders' probabilities for selecting the proposers. Both $p=q=(1/3,1/3,1/3)$ and $p=(1,0,0)$ and $q=(0,x,1-x)$ with $x \in [0,1]$ constitute Nash equilibria. Thus, a continuum of Nash equilibria exists, contrary to the previous section, where only two Nash equilibria emerge in the case of identical offers. 

Due to these intricacies, we refrain from explicitly deriving parameter regions and all possible Nash equilibria for the responders. Instead, we restrict our analysis to Nash equilibria for the responders that are evolutionarily stable. For that, we consider an infinite population of responders from which we choose the respective number of players in each subgame. Thus, the players cannot coordinate. In Theorem 2 (see 
the Supplementary Information, "General Case"), we prove that for any combination of offers $(s_1,s_2,\dots, s_K)$, apart from purely zero offers ($s_i=0,$ for all $i \in \{1,2,\dots,K\}$), there exists a unique evolutionarily stable strategy for the responders. 
We will consider this responder strategy to be the "superior" Nash equilibrium, in the sense of {\it evolutionary stability}. Now we will show some implicit results about the evolutionarily stable Nash equilibrium.

Without loss of generality, let $ s_K \geq s_{K-1} \geq \cdots \geq s_1$ with $s_K > 0 $ represent the offers from $ K \geq 2$ proposers. Assume that all responders employ the evolutionarily stable strategy $\{p_i\}_{i=1}^K$. When including the same strategy $ p_{li} = p_i $ for all $ l \in \{1, 2, \ldots, L\}, $ $i \in \{1,2,\dots,K\}$ in the derivative (see~\eqref{eq:resp_derivative}), we obtain the same derivative of the payoff for all responders (denoted $\Pi_R$):
\begin{equation}\label{eq:der_r_simplified_GC}
\begin{aligned}
    \frac{\partial \Pi_R}{\partial p_i}&= \frac{1}{L} (s_i f(p_i)-s_K f(p_K))\ ,
    \end{aligned}
\end{equation}
where $i \in \{1,2,\dots,K-1\}$ and $f(x)$ is defined as:  
\begin{equation*}
\begin{aligned}
            &f: [0,1] \to \mathbb{R}\ , \quad f(0)=L\ , \\
           & f(x)= L \sum_{j=1}^{L} \frac{1}{j} 
 {L-1\choose j-1}x^{j-1}(1-x)^{L-j}=\frac{1-(1-x)^L}{x} \ .
 \end{aligned}
\end{equation*}
The function $f$ is continuous and strictly decreasing on $[0,1]$ (see Lemma 1 in the Supplementary Information). Since $s_K$ is the highest offer, $p_K$ has to be greater than zero in the Nash equilibrium, which means no other $p_i$ can be equal to one. Therefore, all derivatives must be non-positive in the Nash equilibrium. Then, for the evolutionarily stable strategy $\{p_i\}_{i=1}^K$ it has to follow for all $i \in \{1,2,\dots, K-1\}$ that:  
\begin{equation}
\begin{aligned}
    &  p_i=0  &\text{and }\  &s_i f(p_i)-s_K f(p_K)\leq 0 \ ,\\
    \text{or}\\
    & p_i>0   &\text{and } \ &p_i=f^{-1}\left(\frac{s_K}{s_i}f(p_K)\right) \ .
\end{aligned}\label{eq:eq_rules}
\end{equation}
Note, it is clear that when $s_i=0,$ $p_i$ has to be zero too.

Next, we show that for any offer combination (apart from pure zero offers) there exists a unique solution that follows~\eqref{eq:eq_rules}. 
Since the function $f$ is strictly decreasing on $[0,1]$ it is also invertible on $[0,1].$ We define the function $h$ that describes the sum of the $p_i$'s that follow~\eqref{eq:eq_rules} for a fixed $p_K:$ 
\begin{equation}\label{eq:g_func}
            h: [0,1] \to \mathbb{R}, \quad
            h(p_K)=p_K+\sum_{i=1}^{K-1} f_0^{-1}\left(\frac{s_K}{s_i}f(p_K)\right),
\end{equation}
where $f_0^{-1}(x)=f^{-1}(x)$ on $[1,L]$ and zero on $[L,+\infty].$ To find the equilibrium strategy we need to find $p_K$ such that $h(p_K)=1.$ The function $h$ is strictly increasing in $p_K.$ For $p_K=0: h(0)=0,$  for $p_K=1$ we have $h(1)=1$ if $s_K \geq Ls_{K-1}$ and $h(1)>1$ otherwise. Since $h$ is continuous and monotone there has to be a unique value that satisfies $h(p_K^*)=1.$ This solution is a Nash equilibrium and as we will show in the Theorem 2 (see the Supplementary Information, "General Case") it is also an evolutionarily stable strategy.

For a large number of proposers and responders it is not possible to find a closed-form solution of the responders' evolutionarily stable strategy (numerically it is possible e.g.~by utilizing the $g$ function in~\eqref{eq:g_func}), however, this does not prevent us from deriving the Nash equilibria of the proposers in the next section.
\subsection*{Proposers' Strategy}
In this section, we analyse the proposers' subgame-perfect Nash equilibria. We assume that for each offer regime, all responders play with a unique evolutionarily stable strategy, found as a solution to~\eqref{eq:eq_rules}. 

First, let us comment on the situation where all offers are zero; in this situation, the responders receive zero no matter what strategy they choose. If the number of proposers is bigger than the number of responders there exists no subgame-perfect Nash equilibrium. However, if $K\leq L$ we have degenerate equilibria where each of the $K$ responders have to choose one of the $L$ proposers with probability one and the other $K-L$ responders may choose an arbitrary strategy. Then, all proposers earn a payoff one and cannot improve further.
When the responders opt for strategies that lead to at least one proposer having a probability of selection less than one, this cannot be a subgame perfect Nash equilibrium. The proposer who earns less than one could offer a sufficiently small amount $\varepsilon$, leading to all responders selecting them and resulting in an improved payoff of $1-\varepsilon$.  
\begin{figure*}[ht]
    \centering
      \includegraphics[scale=0.43,trim={4cm 8cm 4cm 8cm},clip]{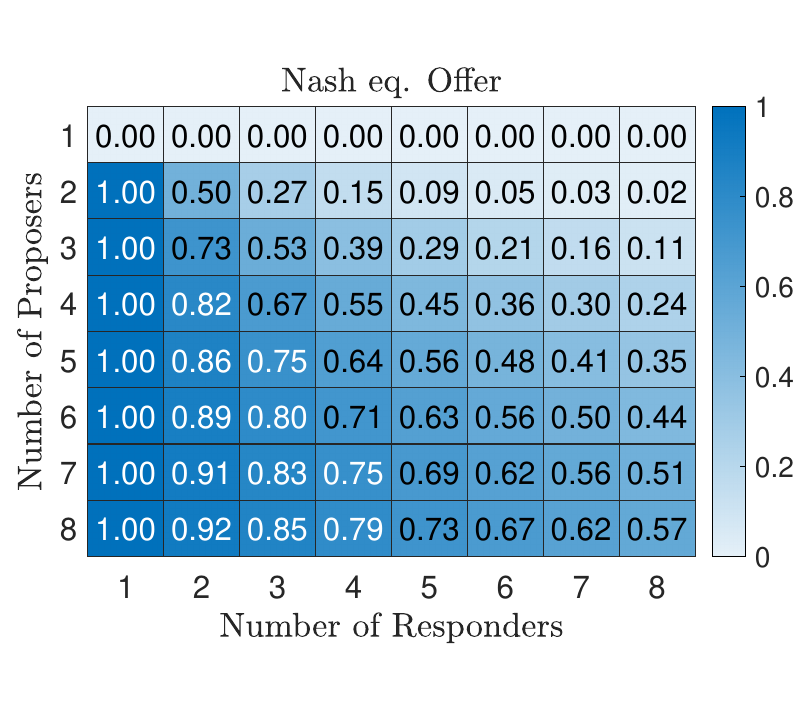}
    \includegraphics[scale=0.43,trim={4cm 8cm 4cm 8cm},clip]{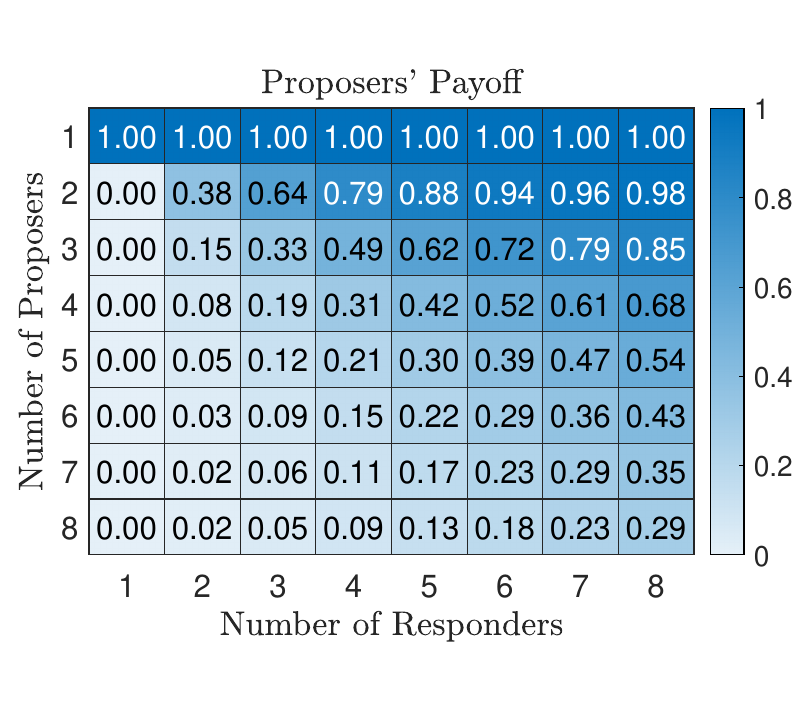}
     \includegraphics[scale=0.43,trim={4cm 8cm 4cm 8cm},clip]{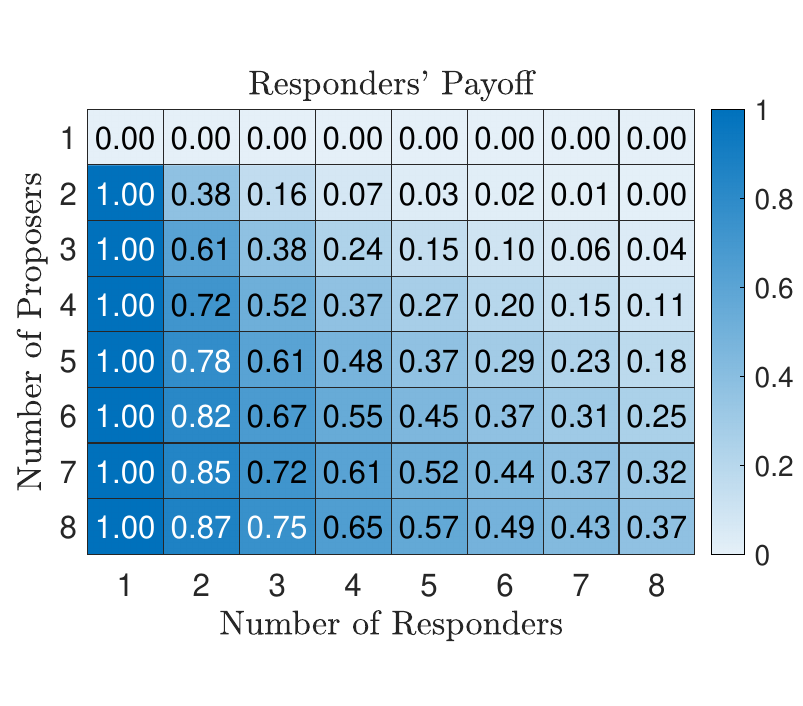}
    \caption{Numeric values of the proposers' Nash equilibria offers from~\eqref{eq:final_Nash} (left) and the expected payoffs for the proposers (middle) and the responders (right) under these equilibria (see~\eqref{eq:exp_payoffs_final}). These are shown for varying number of responders and proposers, under the assumption that responders play their evolutionarily stable strategy and thus select each proposer with the same probability when the offers are identical. }
    \label{fig:EP_proposers_responders}
\end{figure*}
Next, we focus on offer regimes where the highest offer is bigger than zero, i.e.~where $s_K>0$. Offer regimes where some of the proposers end up with zero selection probabilities under responders' evolutionarily stable strategy cannot constitute a subgame-perfect Nash equilibrium. The rationale behind this is straightforward: a proposer with a zero probability of being selected will earn no reward, and by simply offering a sufficiently higher amount, any proposer can secure a positive payoff. If we had for any $i:$ $\frac{\partial \Pi_{R}}{\partial p_{i}} < 0$ (see~\eqref{eq:der_r_simplified_GC} and~\eqref{eq:eq_rules}), then this would give a zero selection probability for proposer $i$. Thus, we can restrict our search only to those regimes where the system of equations $\frac{\partial \Pi_{R}}{\partial p_{i}}$ from~\eqref{eq:der_r_simplified_GC} is equal to zero for all $ i \in \{1,2,\dots,K-1\}$ and $p_i\neq 0$ for all $i \in \{1,2,\dots,K\}.$ Then it is true  for all $i \in \{1,2,\dots,K-1\}$ that:
\[C:= \frac{s_i}{p_i}\left(1-(1-p_i)^L\right)=\frac{s_K}{p_K} \left(1-(1-p_K)^L\right)\ .\]
Additionally, we know that the payoff of proposer $i$ with offer $s_i$ is given as:
\begin{equation*}
\begin{aligned}
\Pi_{P,i}=(1-s_i)\left(1-(1-p_i)^L\right) \ .
\end{aligned}
\end{equation*}
In order to find the subgame-perfect Nash equilibrium of proposers (for details see Theorem 3 in the Supplementary Information, "General Case") 
we look at the derivative:
\begin{equation}
\begin{aligned}\label{eq:payoff_der}
\frac{\partial \Pi_{P,i}}{\partial s_i}
&=\left(\frac{g_i^2C}{s_i^3H}-\frac{g_iC}{s_i^2}\right)\left(L(1-p_i)^{L-1}-C \right)-\frac{p_ig_iC}{s_i^2H} \ ,
\end{aligned}
\end{equation}
where $g_i=\frac{p_i^2}{(1-p_i)^{L-1}(Lp_i+1-p_i)-1}$ and  $H=\sum_{j=1}^K \frac{g_j}{s_j}$. In the equilibrium, the derivative in~\eqref{eq:payoff_der} has to be equal to zero for all $i \in \{1,2,\dots,K\}.$ In Theorem 3 (see the Supplementary Information) we show that such solutions must be symmetric, i.e.~$s_i=s, s \in (0,1)$ for all $i \in \{1,2,\dots, K\}.$ In that case also the selection probabilities of responders are symmetric $p_i=p=\frac{1}{K}$ for all $i \in \{1,2,\dots,K\},$ since this solution satisfies the conditions for the unique responders' ESS strategy.
Taking all this into account, solving the system in~\eqref{eq:payoff_der} being equal to zero is equivalent to solving:
\begin{equation*}  
\begin{aligned}
0&=\left(\frac{g}{s}-H\right)\left(L(1-p)^{L-1}-C \right)-p \ ,\\
\end{aligned}
\end{equation*}
where $g=\frac{p^2}{(1-p)^{L-1}(Lp+1-p)-1}$,  $H=K\frac{g}{s}$ and $C=\frac{s}{p}\left(1-(1-p)^L\right)$. From this we can derive the unique solution $s^*:$
\begin{equation*}
\begin{aligned}
s^*&=\frac{g(1-K)L(1-p)^{L-1}}{g(1-K)\frac{1-(1-p)^L}{p} +p}\ .
\end{aligned}
\end{equation*}
By submitting $p=\frac{1}{K}$ we get:
\begin{equation}
\begin{aligned}
s^{*}&=\frac{L (K - 1)^{L}}{K^{L + 1}-(K - 1)^{L-1} ((K - 1) K + L) }.
\end{aligned}
\label{eq:final_Nash}
\end{equation}
It is easy to show that $s^{*} \in (0,1)$ for all $K,L\geq2$ and that the second derivative of~\eqref{eq:payoff_der} is always negative (see Proposition 1 in the Supplementary Information). Therefore, $s^*$ in~\eqref{eq:final_Nash} is indeed a subgame-perfect Nash equilibrium under the assumption that responders behave according to the evolutionarily stable strategy. 
Apart from the degenerate case of all offers being equal to zero, this is the only subgame-perfect Nash equilibrium. The expected payoff of proposers and responders is:
 \begin{equation}\label{eq:exp_payoffs_final}
     \begin{aligned}
         \Pi_P &= (1-s^*)\left(1-\left(1-\frac{1}{K}\right)^L\right),\\
         \Pi_R &= s^*\frac{K}{L}\left(1-\left(1-\frac{1}{K}\right)^L\right).
     \end{aligned}
 \end{equation}

Next, we analyse the result in~\eqref{eq:final_Nash} and investigate how the competition (im)balance of both sides of the market influences the thresholds $s^*$. Numeric values of equilibrium offers and equilibrium payoffs for a small number of proposers and responders are shown in Fig.~\ref{fig:EP_proposers_responders}. 
One can notice, that for a fixed number of proposers ($\geq 2$) and an increasing number of responders, i.e.~increasing responder competition, the equilibrium offers decrease. On the other hand, for a fixed number of responders ($\geq 2$) and increasing proposer competition, the proposers' offers increase. This observation aligns with intuitive expectations about the effect of competition. 
The expected payoffs for responders and proposers follow the same trend. The proposers achieve the highest expected payoff when a single proposer faces any number of responders. Conversely, a responder gains the most when they are the sole responder facing at least two proposers.
Subsequently, we shall examine the Nash equilibrium strategy and payoffs of both proposers and responders in the asymptotic scenario where $K \gg 1$ and $L \gg 1$ and the ratio of the number of proposers to the number of responders is determined by a constant factor $c>0,$ i.e.~$K=cL.$ A low value of $c$ means there is strong responder competition, high $c$ implies strong proposer competition. Then the equilibrium $s^*$ depends on $c$ as:
\begin{equation}
    \begin{aligned}
      s^*_c = \frac{1}{c(e^\frac{1}{c} - 1)}\ .
    \end{aligned}\label{eq:limit_offer}
\end{equation}
The payoff of the proposers and responders in this asymptotic situation is: 
  \begin{equation}
     \begin{aligned}
        \lim_{L\to \infty } \Pi_P &= 1 - \frac{c + 1}{e^\frac{1}{c}c} \ ,      \quad
         \lim_{L\to \infty } \Pi_R &= \frac{1}{e^\frac{1}{c}} \ .
     \end{aligned}
     \label{eq:limit_payoffs}
 \end{equation}

In Fig.~\ref{fig:c_dependence} we plot the results from ~\eqref{eq:limit_offer} and~\eqref{eq:limit_payoffs}. One can see how variations in the competition factor $c$ impact the offers and payoffs in the limit.
\begin{figure}
    \centering
  \includegraphics[scale=0.5,trim={3cm 8.6cm 2.5cm 8.6cm},clip]{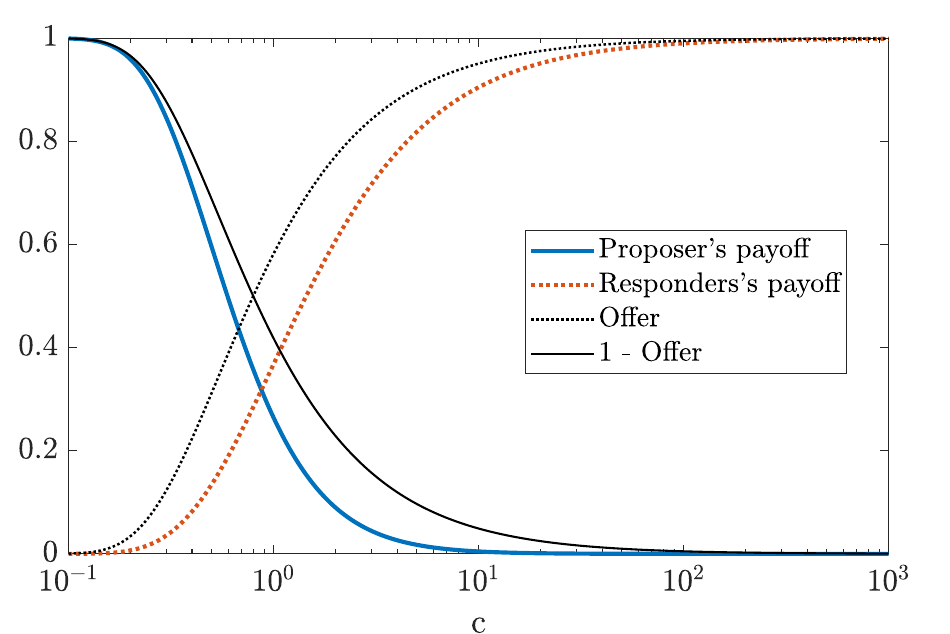}
    \caption{The expected payoff for proposers (blue), for responders (red dotted), the offer level (black) and 1-the offer level (black dotted) in subgame-perfect Nash equilibrium with respect to the proposer-responder ratio $c$, i.e.~$K=cL$, in the large $L$ limit 
    (compare~\eqref{eq:limit_offer} and~\eqref{eq:limit_payoffs}). The differences between the offer and expected payoffs arise from "inefficiencies". That is, with some probability some responders choose the same proposer and some proposers may not get chosen by any responder.}
    \label{fig:c_dependence}
\end{figure}
In the next section, we look deeper into our analysis and explore possible implications of the MPMR UG on fairness.
\section*{Discussion}\label{sec:discussion}
Throughout our evolutionary history, the interactions within social communities have played a crucial role in shaping the behaviour of our species. Many of these interactions were not interactions in pairs but rather in groups, suggesting that the MPMR UG with multiple players on both sides can contribute to the ongoing debate on \textit{fairness} extending beyond the insights provided by the one vs.~one UG. 
In this section, we will discuss some of these potential implications. 

The perception of what is a \textit{fair} or \textit{equitable} division might be impacted by the properties of the community one is part of. For example, individuals with a particular skill or opportunity might represent the proposers, while unskilled individuals, who are still essential for task completion, are the responders. Our model suggests that if responders perceive a shortage of skilled individuals, they may place greater value on this expertise and, consequently, find it \textit{fair} to accept a smaller share of the reward in exchange for assisting the proposer.

Then a central question arises: what is the underlying sense of group size balance between proposers and responders (i.e.~proposer-responder ratio $c$) in the (local) community? And consequently, what would responders deem a fair share of the reward in such a situation? A basic initial assumption could be a balanced scenario where the number of proposers equals that of responders, i.e.~when $c=1$. In this case, the subgame-perfect Nash equilibrium offer $s^*_1$ is equal to $\approx 0.582,$ which is slightly above the equality threshold of $0.5.$ The expected payoff of responders is $\approx 0.368$ and of proposers is $\approx 0.264$ (see~\eqref{eq:limit_offer} and~\eqref{eq:limit_payoffs}). When a responder is directly approached by a proposer (resembling the one vs.~one UG), they may deem it fair to receive a share of the reward equivalent to what they would obtain in the community setting of the MPMR UG, i.e.~36.87\% of the reward. 

However,  when $c=1$, a proposer receives a smaller payoff than a responder. Drawing from the previous analogy, we might anticipate a decrease in the number of individuals specializing in the skill until the payoff difference is eliminated. 
Such an "equity" ratio $c$ when the payoffs for members of both groups are identical is $c=0.872,$ meaning there are fewer proposers than responders.
The payoff of each proposer and responder is approximately $0.318$ and the proposers' offers are around $0.534.$ Once again, a responder could deem it fair to get the same payoff of $0.318$ from the one vs.~one interaction with a proposer as they would enjoy in the MPMR UG scenario with many players.
 
Note if there is some cost the proposer has to pay in order to be able to split and share the reward (e.g.~energy and time invested for training to acquire a skill, or alternatively, to find an opportunity) their actual payoff is lower accordingly. This moves the blue curve of the proposers' payoff in Fig.~\ref{fig:c_dependence} down, changing the equity ratio to a lower number, leading to an even smaller proportion of proposers in the community and finally smaller rewards for the responders in the equity state. 

Next, let us compare the derived thresholds with the literature on the UG and the Dictator Game in which the responder cannot refuse the split and gets what is offered by the proposer. 

As an indication of fairness norms, one might look at the outcomes of the Dictator Game. Even though their offers cannot be refused, proposers in experiments often give a positive reward to the responder. In a meta-analysis by Engel~\cite{engel2011} the mean offer in the Dictator Game was $28\%$ with a modal (i.e.~typical) offer of zero among W.E.I.R.D. societies. Conversely, in three small-scale societies, Henrich et al.~\cite{henrich2005} found mean offers of 20\%, 31\%, and 32\% percent and only a few subjects offered zero. Another recent meta-analysis by Cochard et al.~\cite{cochard2021} supports the results with a calculated mean offer of 30.6\% in the Dictator Game. We can notice these values are very close to our estimated payoff of the responder in the balanced case (31.8\%). 

In a meta-analysis of UG studies, Oosterbeek et al.~\cite{oosterbeek2004} found a mean offer of $40\%$ and Cochard et al.~\cite{cochard2021} found the mean to be $42.58\%.$ 
Henrich et al.~\cite{henrich2005} conducted experiments within small-scale societies and found the mean offers ranging from 26 \% to 58\% for different societies.
One can ask whether proposers' behaviour can be attributed to fairness concerns or simply strategic behaviour, wherein they offer proposals that maximise their payoffs given the assumed distribution of acceptance among responders. For example, Roth et al.~\cite{roth1991} found evidence supporting strategic behaviour among proposers. 
On the other hand, Henrich et al.~\cite{henrich2005} concluded that offers tend to be higher than the optimal payoff-maximizing offer, presumably due to pessimism regarding rejection frequencies and ambiguity aversion. 
These payoff maximizing offers are found to be around 25\%-40\% for different small-scale societies. Nevertheless, Henrich et al.~also noted instances where certain groups exhibited a tendency to accept nearly all low offers, while other groups commonly rejected high offers. Furthermore, Solnick~\cite{solnick2001} found the average minimally acceptable offer (the offer below which the responder rejects the offer) to be 30.8\%. Once again, the thresholds are quite similar to the ones proposed from the MPMR UG.

Other authors, e.g.~Schuster~\cite{schuster2017}, claim the Golden Ratio, of about 0.618 vs.~about 0.382, to be the solution for fair thresholds and supports this claim by providing numerous examples from the literature that illustrate real-life situations exhibiting similar patterns. This threshold is once again close to ours, even though the author describes different mechanisms for achieving it.
\section*{Conclusions}\label{sec:conclusions}
We have proposed a multi-player version of the UG, which we named the Multi-Proposer-Multi-Responder Ultimatum Game, with multiple responders and proposers playing simultaneously. 
Our work offers a new perspective on the UG with the interplay between proposer and responder competition. We analysed the responders' strategy patterns and found that there can be a continuum of Nash equilibria of responders for some subgames (determined by offers from proposers), however, there is only one unique evolutionarily stable strategy in each subgame. We analytically derive subgame-perfect Nash equilibria of the proposers with respect to this strategy, and find that situations with multiple proposers and responders in both groups lead to non-trivial offer values. These are unique for each parameter regime (except for the \textit{degenerate} solutions of all offers being zero). We stress that our model does not include bargaining sequential dynamics to reach non-trivial offer levels and we achieve this under a one-shot setting. Admittedly, the game's assumptions are heavily simplified when compared to complex economic and societal realities. We consider all goods (or jobs, flowers, goods etc.) to be of identical quality, all with the possibility to be reached in the same way and we assume all responders have full information of all the proposals. Yet, we believe there is an untapped potential in examining notions of fairness within the context of larger groups of players. Since real-world interactions rarely involve just two individuals, and the presence of external alternatives remains pertinent, we believe this expanded viewpoint can shed more light on our understanding of why fairness is perceived as it is. 

Our analysis has revealed the potential impact of the ratio between the (asymptotic) number of proposers and responders on the perception of fairness in multi-player scenarios. We postulate that the sense of fairness within these scenarios then subconsciously influences people's behaviour in the UG laboratory experiments. In particular for the UG played on social networks, MPMR UG-like situations happen naturally in the neighbourhood of individuals. Conversely, the MPMR UG can explain the incentives for the way social networks evolve and restructure. 
We compared these insights with the experimental literature of one vs.~one UG and Dictator Game and found them to be of intriguing similarity. 

We stress we do not claim the MPMR UG to be the sole explanation of the observed behaviour, rather, we suggest it as a noteworthy addition to the ongoing discussion on fairness.
\begin{acknowledgments}
We would like to thank B\'{a}lint Homonnay for his comments and helpful discussion. This project has received funding from the European Union’s Horizon 2020 research and innovation programme under grant agreement No. 955708.
\end{acknowledgments}
\bibliography{references}
\end{document}


\title{Supplementary Information: Fairness in Multi-Proposer-Multi-Responder Ultimatum Game}
\author{\orcidlink{0000-0002-0062-3377}Hana Krakovsk\'{a}}
\email[Corresponding author:]{hana.krakovska@meduniwien.ac.at}
\affiliation{Institute of the Science of Complex Systems, Center for Medical Data Science, Medical University of Vienna, Spitalgasse 23, Vienna, 1090, Austria}
\affiliation{Complexity Science Hub, Josefstädter Str. 39, Vienna, 1080, Austria}
\author{\orcidlink{0000-0002-4045-9532}Rudolf Hanel} 
\affiliation{Institute of the Science of Complex Systems, Center for Medical Data Science, Medical University of Vienna, Spitalgasse 23, Vienna, 1090, Austria}
\affiliation{Complexity Science Hub, Josefstädter Str. 39, Vienna, 1080, Austria}
\author{\orcidlink{0000-0002-1698-5495}Mark Broom}
\affiliation{Department of Mathematics, City, University of London, Northampton Square, London, EC1V 0HB, United Kingdom}
\maketitle

In this Supplementary Information, we will prove and show several key theorems and observations that are utilized in the main part of the manuscript. In the following section we start by proving that two vs.~two Multi-Proposer-Multi-Responder Ultimatum Game (MPMR UG) has a unique evolutionarily stable strategy. 
We also demonstrate that under replicator dynamics with a population composed of players playing Nash equilibria, the population in a stable state plays the derived evolutionarily stable strategy, albeit on a population-wide level. 
In  "General Case", we extend the results from two vs.~two MPMR UG to arbitrary numbers of multiple responders and multiple proposers and show that also, in this case, there is a unique evolutionarily stable strategy for any subgame (apart from pure zero offers). Furthermore, we prove that if responders play this ESS Nash equilibrium in each subgame then in the subgame-perfect Nash equilibrium proposers have to propose the same offers.
\phantomsection
\section*{Two Proposers and Two Responders}\label{app:section_two_two}
In this section, we will prove results concerning two vs.~two MPMR UG (refer to the section with the same title in the main manuscript). 
We will prove that the symmetric strategy 
\begin{equation}\label{eq:pA}
     p_A = \frac{2s_1 - s_2}{s_1 + s_2} \ ,
\end{equation}
where $s_1, s_2$ are offers which are similar enough ($s_2<2s_1$), is an \hyperref[appendix:ESS_two_two]{evolutionarily stable strategy}. Then in 
\hyperref[appendix:repliator]{Replicator Dynamics} we demonstrate that under replicator dynamics with three different Nash equilibrium strategy types, the solutions on the stable line of equilibria form the mixed strategy $p_A$, albeit on a population level.
\subsection*{Evolutionarily Stable Strategy}\label{appendix:ESS_two_two}
First, we start with the definition of an evolutionarily stable strategy.
\begin{defin}\label{def:2x2ESS}
    Let us denote by $\Pi(X,Y)$ the payoff of a player playing strategy $X$ against strategy $Y.$ Strategy $X$ is called evolutionarily stable if for any $X\neq Y$:\
\begin{enumerate}
\item $ \Pi(X,X)>\Pi(Y,X)$ \ ,
\end{enumerate}
or 
\begin{enumerate}
\item[2.]$ \Pi(X,X)=\Pi(Y,X) \text{ and } \Pi(X,Y)>\Pi(Y,Y) \ .$
\end{enumerate}
\end{defin}
\noindent In the following theorem we prove that strategy $p_A$ (see \eqref{eq:pA}) is a unique evolutionarily stable strategy.
\begin{theorem} Consider MPMR UG with two proposers and two responders. Consider proposers' offers, denoted as $s_1, s_2,$ where $0< s_1 \leq s_2\leq 2s_1$ and $ s_2 \leq 1.$ Then, the mixed strategy $(p_A,1-p_A),$ where $\displaystyle{p_A=\frac{2s_1-s_2}{s_1+s_2}}$ is the probability of choosing the first proposer, is a unique evolutionarily stable strategy for each subgame defined by offers $s_1, s_2$.
\label{prop:ESS}
\end{theorem}
\begin{proof}

We denote the strategy $(p_A,1-p_A)$ as $A$ and prove it is evolutionarily stable according to the second condition of Definition~\ref{def:2x2ESS}, we also refer to the notation introduced there. 
Let us rewrite the offers as $s_1=s$ and $s_2=s+\delta,$ where $0\leq \delta \leq s.$  
It is easy to see (see the manuscript for details) that:
\begin{equation*}
    \Pi(A,A)= \frac{3s(s + \delta)}{2(2s + \delta)} \quad \text{and} \quad 
    \Pi(Y,A)= \frac{3s(s + \delta)}{2(2s + \delta)} \quad \text{for any strategy } Y\ ,
\end{equation*}
which proves the first property of the second condition. 
 Next, we prove that
\[\Pi(A,Y)>\Pi(Y,Y) \text{ for any strategy }Y\neq A\ .\]  
Let us represent strategy $Y$ as having the probability $p \in [0,1]$ of going to the first proposer and probability $1-p$ of going to the second proposer. Then
\begin{equation}
    \begin{aligned}
    \Pi(A,Y)-\Pi(Y,Y)&= \frac{(\delta - s + \delta p + 2ps)^2}{2(2s+\delta)}\ , 
    \end{aligned}
\end{equation} 
where the equivalence is reached when
\[ \delta - s + \delta p + 2ps=0 \quad  \iff \quad p=\frac{s-\delta}{2s+\delta}\ . \] Since strategy $\displaystyle{p=\frac{s-\delta}{2s+\delta}}$ is in fact strategy $A,$ we showed that $\Pi(A,Y)>\Pi(Y,Y)$  for any strategy $ Y\neq A.$ Additionally, strategy $A$ is uniquely ESS, since for any strategy $X\neq A,$ we can not have $\Pi(X,X)\geq \Pi(Y,X)$ because $ \Pi(X,X)<\Pi(A,X).$ 
\end{proof}
\subsection*{Replicator Dynamics}\label{appendix:repliator}
In this subsection, we analyse the replicator dynamics involving three types of responders, each corresponding to a different Nash equilibrium strategy. We demonstrate that under the replicator dynamics the abundance ratios in the stable states correspond to the strategy $p_A$ (see~\eqref{eq:pA}), albeit on a population level.

Let us consider three types of players: $\gamma$--players that always play evolutionarily stable strategy $p_A$, $\alpha$--players who always choose the higher offer, $\beta$--players who always choose the smaller offer. If the offers are the same, $\alpha$ and $\beta$ players choose randomly one of them. We denote the fractions of $\gamma$--players, $\alpha$--players, and $\beta$--players in the population as $x_1$, $x_2$, and $x_3$, respectively, where $x_1 + x_2 + x_3 = 1.$ We use the same notation as in Theorem~\ref{prop:ESS}, where the offers are denoted as $s_1$ and $s_2,$ with $s_1 = s$ and $s_2 = s + \delta$, where $0 \leq \delta \leq s.$ Using the payoff matrix presented in Table 1 in the main manuscript, denoted here as $P$, we can construct the equations of the replicator dynamics:
    \begin{equation*}
        \dot{x}_i=x_i\left[(Px)_i-x^TPx\right] \quad \text{for }i \in \{1,2,3\} \ ,
        \label{eq:replicator_general}
    \end{equation*}
    which gives us 
        \begin{equation}
        \begin{aligned}
         \dot{x}_1&=x_1\frac{(x_1(2\delta +s) + x_2(\delta + 2s) - 2\delta -s )^2}{2\delta + 4s},\\
       \dot{x}_2&=x_2\frac{\splitfrac{\delta^2(4x_1^2 + 4x_1 x_2 - 6x_1 + x_2^2 - 3x_2 + 2) + \delta s( 4x_1^2 + 10 x_1x_2 - 9x_1+}{ 4 x_2^2 - 9 x_2 + 5) + s^2(x_1^2 + 4x_1 x_2 - 3 x_1 + 4 x_2^2 - 6 x_2 + 2)}}{2\delta + 4s}\ ,
       \label{eq:replicator}
     \end{aligned}
     \end{equation}
     where $\dot{x}_3=-\dot{x}_1-\dot{x}_2.$
     
Let us analyse the stability of the system.
     It is easy to show that there are two equilibria $\{0,0,1\},$  $\{0,1,0\}$ and a line of equilibria ${l^*=\{1-x_2\frac{\delta + 2s}{2\delta + s},x_2,x_2\frac{s-\delta}{2\delta + s}\}}$.
     The linearisation of the system at equilibria $\{0,0,1\}$ and $\{0,1,0\}$ reveals that both equilibria are unstable.

     Linearisation at the line reveals that the equilibria have one zero and one negative eigenvalue. As seen in the numerically analysed dynamics (see Figure~\ref{fig:replicator_simplex}) the equilibrium line $l^*$ is globally attracting on the whole simplex apart from the two unstable equilibria.
    
    In the stable equilibrium line case $l^*$ the expected payoffs of responders are:
    \begin{equation}
        \begin{aligned}
        \Pi(\alpha)=\Pi(\beta)= \Pi(\gamma)= \frac{3s(\delta + s)}{2(\delta + 2s)}\ ,
     \end{aligned}
     \end{equation}
    and the expected probability of going to the lower offer proposer in the mixed case on the line is
   $\frac{s-\delta}{\delta +2s}$ 
    which is the same probability as in strategy $p_A.$ 

\begin{figure}
    \centering
      \includegraphics[scale=0.45,trim={3cm 3cm 3cm 3cm},clip]{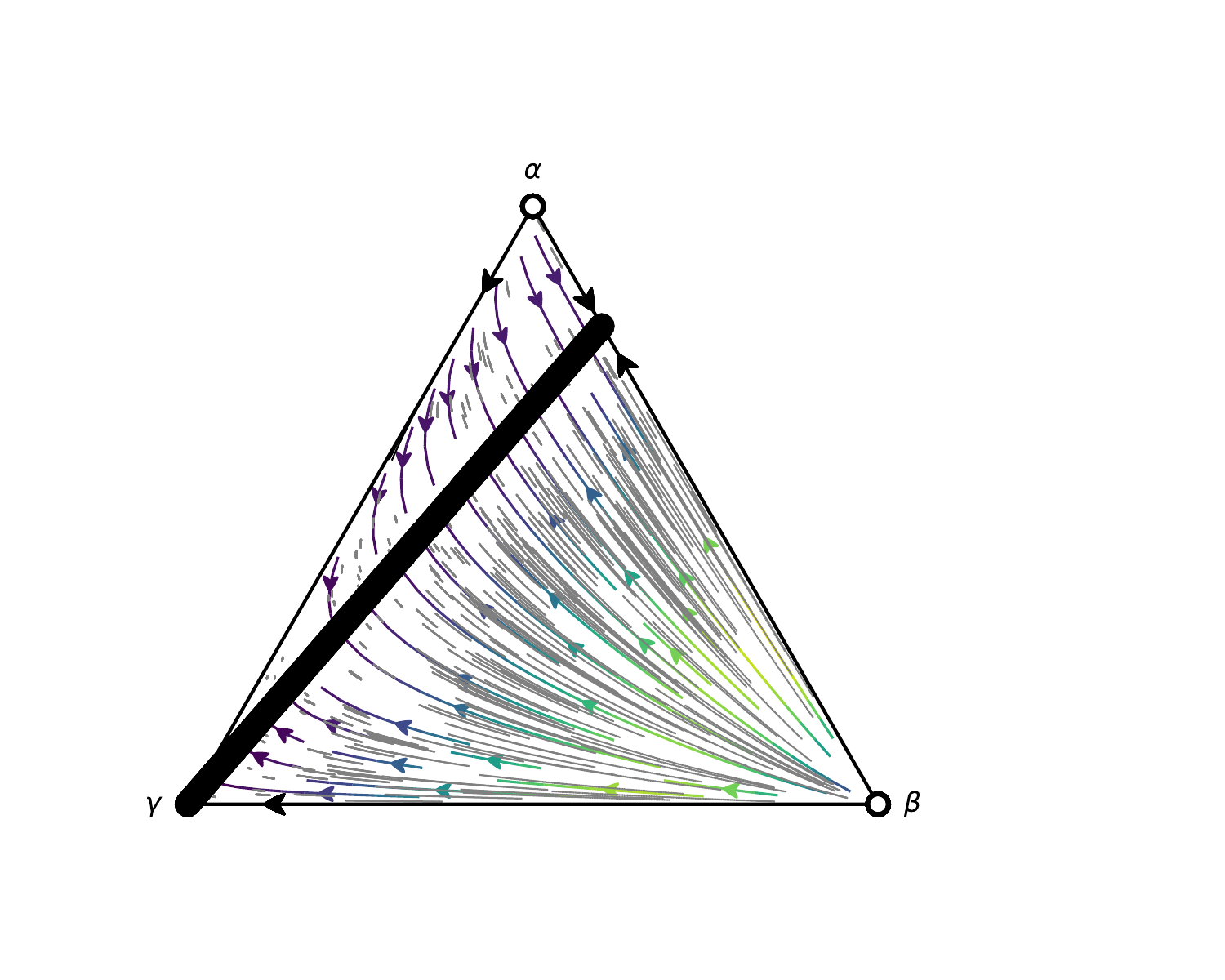}
    \includegraphics[scale=0.45,trim={3cm 3cm 4cm 3cm},clip]{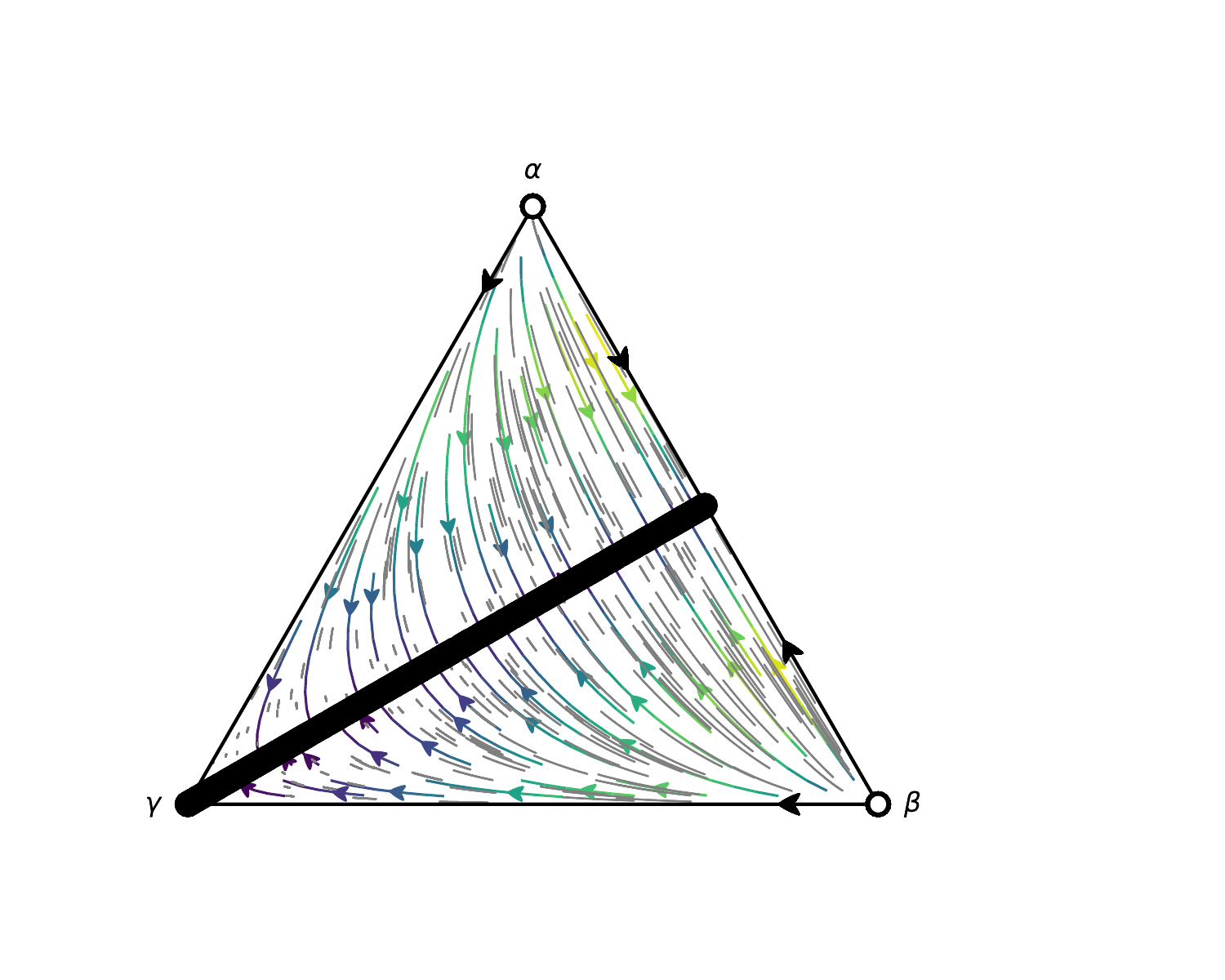}
    \caption{Phase-space of the replicator dynamics (see~\eqref{eq:replicator}) for $s=0.2,\delta=0.1$ (left) and $s=0.2,\delta=0$ (right). Light yellow stands for high gradient values, and deep purple for low. Figures were generated with software made by Fernández~\cite{fernandez2020}.}
    \label{fig:replicator_simplex}
\end{figure}

\section*{General Case}
\label{app:general}
In this section we provide proofs of propositions and theorems needed in the main manuscript (section named "$K$ Proposers and $L$ Responders"). In the first subsection 
we show that there exists a unique evolutionarily stable strategy for any offer combination, apart from all offers being zero. In the second subsection 
we prove that if responders play this evolutionarily stable strategy, then the proposers' equilibrium strategy must be symmetric.

We start by repeating the notation from the main manuscript. We consider $K\geq 2$ proposers and $L\geq 2$ responders. The strategy of each proposer that describes how much they offer to the responders is denoted as $s_i \in [0,1],$ $ {i \in \{ 1,2,\dots,K\}},$ and for every responder we allow a mixed strategy denoted:
\begin{equation*}
 p_{l,i}\geq 0,\ l \in \{ 1,2,\dots,L\}, \ i \in \{ 1,2,\dots,K\}, \text{ s.t.~for all $l$}: \sum_{i=1}^K p_{l,i} =1\ ,
\end{equation*}
which determines the probability of the responder $l$ choosing the proposer $i.$ Note that the responder may reject all offers with some probability. In this case, the sum of the visitation probabilities does not necessarily equal one. However, to identify Nash equilibria of responders in each subgame, we can exclude these strategies if at least one offer is greater than zero.
Let us denote $\Pi_{R,l}$ the expected payoff of responder $l.$ We call the following system of equations \textit{System R}:

For $l \in \{1,2,\dots, L\}$ and $i \in \{1,2,\dots,K-1\}:$
  \begin{equation}
\begin{aligned}
    &\frac{\partial \Pi_{R,l}}{\partial p_{li}} =  s_i W_{l,i}-s_K W_{l,K}\ ,\\
    &\text{where for all } m \in \{1,2,\dots,K\} :  
    &W_{l,m}=\sum_{j=1}^{L}  \frac{1}{j} \left[ \sum_{\alpha:\norm{\alpha}_{1}=j,\alpha_l=1} \prod_{k=1,k\neq l}^{L} p_{km}^{\alpha_k}(1-p_{km})^{1-\alpha_k}\right].
\end{aligned}
\label{eq:appendix_system_R}
\end{equation}

If all other responders apart from responder $L$ use the same strategy $p_i, i \in \{1,2,\dots,K\},$ where  $\sum_{i=1}^K p_i=1$ and $p_i\geq 0,$ then the derivatives of the payoff of responder $L$ can be simplified to \textit{System S}:

For $i \in \{1,2,\dots,K-1\}:$
  \begin{equation}\label{app:system_s}
  	\begin{aligned}
    &\frac{\partial \Pi_{R,L}}{\partial p_{Li}}= \frac{1}{L} (s_i f(p_i)-s_K f(p_K))\ ,\\
         & \text{where }    f: [0,1] \to \mathbb{R}, \quad
          f(x)=\frac{1-(1-x)^L}{x} \quad \text{and} \quad f(0)=L\ .
    \end{aligned}
\end{equation}

\subsection*{Evolutionarily Stable Strategy}\label{app:subs_ess_general}
In this section we prove Theorem~\ref{thm:ESS} which states that for any set of proposer offers (apart from all proposers offering zero) there exists a unique evolutionarily stable strategy of responders. We refer to the definition of evolutionarily stable strategy in multiplayer games by Broom et al.~\cite{broom2022}.
\begin{defin}\label{def:multi_ESS}
	Consider a multiplayer game with $m$ players. Strategy $A$ is evolutionarily stable against strategy $Y$ if and only if there is $j \in \{0,1,\dots,m-1\}$ such that
	\begin{equation*}
		\begin{aligned}
			\Pi(A;A^{m-1-j},Y^{j})&>  \Pi(Y;A^{m-1-j},Y^{j}), \\
			\Pi(A;A^{m-1-i},Y^{i})&=  \Pi(Y;A^{m-1-i},Y^{i}) \text{ for all } i<j,
		\end{aligned}
	\end{equation*}
 where $\Pi(A;X_1,X_2,\dots,X_{m-1})$ denotes payoff of player playing strategy $A$ against $m-1$ players playing strategies $X_i, i \in \{1,2,\dots,m-1\}$ and  $X^i$ denotes that $i$ players use the same strategy $X.$
Strategy $A$ is called an \textit{evolutionarily stable strategy at level $J$} if, for every $Y\neq A$ the conditions above are satisfied for some $j\leq J$ and there is at least one $Y\neq A$ for which the conditions are met for $j=J$ precisely.	
\end{defin}
First we prove two simple lemmas which will be needed for the proof of Theorem~\ref{thm:ESS}. In the first lemma we show some basic properties of a central function in the derivatives of responders' payoffs (\eqref{app:system_s}). 
\begin{lema}\label{lemma:fun_decreasing}
	Consider $L\geq 2$ and function 
	\begin{equation*}
		f: \mathbb{R} \to \mathbb{R}, \quad
		f(x)=\frac{1-(1-x)^L}{x} \quad \text{and} \quad f(0)=L \ .
	\end{equation*}
	Then $f$ is strictly decreasing, continuous and positive on $[0,1].$ 
\end{lema}
\begin{proof}
	It is easy to show with l'H\^{o}pital rule that $\lim_{x \to 0} f(x)=L$ which gives us that the function $f(x)$ is continuous. Next, we want to show that the function's derivative is always negative on $(0,1]:$
	\[ f'(x)=\frac{L x(1 - x)^{L - 1} + (1 - x)^L - 1}{x^2}\ .\]
	As we prove in Lemma~\ref{lemma:minus_gi_decreasing} function $-f'$ is a positive function, meaning the derivative $f'$ is a negative function, thus $f$ is decreasing, it is also positive on $[0,1],$ since $f(1)$ is $1.$
\end{proof}
In the following lemma, we prove a positive definiteness of a specific matrix, needed for a proof of convexity in Theorem~\ref{thm:ESS}.
\begin{lema}\label{lemma:positive_definite}
Consider $K\times K$ symmetric matrix $H$ with off-diagonal elements $H_{ij}=C, C>0$ for $i \neq j$ and diagonal elements $H_{ii}=C+\varepsilon_i,$ where $\varepsilon_i> 0$ for all $i \in \{1,2,\dots, K\}.$ Then $H$ is positive-definite.
\end{lema}
\begin{proof}
	Matrix $H$ is said to be positive-definite if 
	\begin{equation}\label{def:positive_definite}
		x^THx > 0 \text{ for any vector } 0\neq  x \in \mathbb{R}^K\ .
	\end{equation}
	We have that
	\begin{equation}\label{eq:xhx}
 x^THx= \sum_{i=1}^K (C+\varepsilon_i) x_i^2+  \sum_{i,j=1,i\neq j}^K C x_i x_j\ .	    
	\end{equation}
	It is clear that if $\sum_{i,j=1,i\neq j}^K C x_i x_j \geq 0$ the condition in~\eqref{def:positive_definite} is satisfied. Now, let us assume $\sum_{i,j=1,i\neq j}^K C x_i x_j < 0$ and set $\varepsilon_{min}=\min_{i\in \{1,2,\dots, K\}} \varepsilon_i.$ Then we can rewrite~\eqref{eq:xhx} as
	\begin{equation*}
	x^THx= (C+\varepsilon_{min})\left[\left(\sum_{i=1}^K x_i\right)^2- \sum_{i,j=1,i\neq j}^K x_i x_j \right] + \sum_{i=1}^K (\varepsilon_i-\varepsilon_{min})x_i^2 +\sum_{i,j=1,i\neq j}^K C x_i x_j \ ,
    \end{equation*}
	which is positive since it equals
	\[(C+\varepsilon_{min})\left(\sum_{i=1}^K x_i\right)^2  + \sum_{i=1}^K (\varepsilon_i-\varepsilon_{min})x_i^2 -\varepsilon_{min}\sum_{i,j=1,i\neq j}^K  x_i x_j\ , \]
	and we have $\sum_{i,j=1,i\neq j}^K C x_i x_j < 0 \ .$   
\end{proof}
Finally, we prove the main theorem of this subsection which states that for any offer combination (apart from all offers being zero) there exists a unique evolutionarily stable strategy of responders.
\begin{theorem}\label{thm:ESS}Consider MPMR UG with $K\geq 2$ proposers and $L\geq 2$ responders.  
Consider a set of offers $s_K\geq s_{K-1}\geq \dots \geq s_1,$ where $s_K \in (0,1]$ and $s_1 \geq 0.$ Then for each given set of offers, there exists a unique evolutionarily stable strategy at level $1$ of responders, implicitly defined for all $i \in \{1,2,\dots,K-1\}, p_i\geq 0$ such that $\sum_{i=1}^K p_i=1 $ and $p_K>0$ as:
	\begin{equation}
		\begin{aligned}
			&p_i=0   \ &\text{ and } \  s_i f(p_i)-s_K f(p_K)&\leq 0\ ,  \\
   &or\\
		   &p_i>0  \    &\text{ and }  \   s_i f(p_i)-s_K f(p_K)&=0  \hspace{0.5em}  \Rightarrow  \hspace{0.5em} p_i=f^{-1}\left(\frac{s_K}{s_i}f(p_K)\right), 
		\end{aligned}\label{app:eq_eq_rules}
	\end{equation}
 where  $f: [0,1] \to \mathbb{R}, f(x)=\frac{1-(1-x)^L}{x}$ and $ f(0)=L.$
\end{theorem} 
\begin{proof} Let us denote the solution $p_i$ of~\eqref{app:eq_eq_rules} as strategy $A.$ As was shown in the analysis in "$K$ Proposers and $L$ Responders", such solution always exists and is unique for any set of offers (apart from pure zero offers). 
First, let us assume we have a set of offers that results in strategy $A$  following the equivalence for all $i: s_i f(p_i) - s_K f(p_K) = 0. $
Now refer to the previous Definition~\ref{def:multi_ESS}, for $j=0$ we have to show that
	\[\Pi(A;A^{L-1}) =  \Pi(Y;A^{L-1})\ .\]
This is evident from the fact that the derivatives of the first responder's payoff with respect to their strategies, \(\frac{\partial \Pi_{R,1}}{\partial p_{1i}}\), are independent of their strategy $p_{1i}$ and equal to zero when all other players use strategy $A$ (see \textit{System $S$} in \eqref{app:system_s}).

Next, we look at $j=1.$ We need to show that \[\Pi(A;A^{L-2},Y) >  \Pi(Y;A^{L-2},Y)\ .\] 
Let us denote the strategy $Y$ as $q_i,i\in\{1,2,\dots, K\}.$
We may rewrite the payoffs of the players using strategy $A$ and $Y$ (denoted $\Pi_A$ and $\Pi_Y$ respectively) as:
	\begin{equation*}
		\begin{aligned}
			\Pi_A&=\sum_{i=1}^K s_i p_i\left[ q_i A_{1,L}(i)+(1-q_i)B_{1,L}(i) \right]\ ,    \\    
			\Pi_Y&=\sum_{i=1}^K s_i q_i \left[ q_i A_{1,L}(i)+(1-q_i)B_{1,L}(i) \right]\ ,
		\end{aligned}
	\end{equation*}
	where 
	\begin{equation*}
		\begin{aligned}
			A_{1,L}(i)&= \sum_{j=2}^{L} \frac{1}{j}  \sum_{\alpha:\norm{\alpha}_{1}=j,\alpha_1=1,\alpha_L=1} \quad \prod_{k\neq 1,L}p_{i}^{\alpha_k}(1-p_{i})^{1-\alpha_k}\ ,  \\ 
			B_{1,L}(i) &= \sum_{j=1}^{L-1} \frac{1}{j}  \sum_{\alpha:\norm{\alpha}_{1}=j,\alpha_1=1,\alpha_L=0} \quad \prod_{k\neq 1,L} p_{i}^{\alpha_k}(1-p_{i})^{1-\alpha_k}\ ,
		\end{aligned}
	\end{equation*} 
	and if number of responders is two then 
	\[A_{1L}(i)=\frac{1}{2} \quad \text{and} \quad B_{1L}(i)=1\ .\]
	From this we have
	\begin{equation*}
		\begin{aligned}
			\Pi_A-\Pi_Y&=\sum_{i=1}^K s_i (p_i-q_i)\left[ q_i A_{1,L}(i)+(1-q_i)B_{1,L}(i) \right] .
		\end{aligned}
	\end{equation*}
	Then $\Pi_A-\Pi_Y=0$ for $q=p.$
	From the way the strategy $A$ is defined we have that
	\begin{equation}\label{eq:deriv_zero}
		\begin{aligned}
			s_i \left[ p_i A_{1,L}(i)+(1-p_i)B_{1,L}(i) \right]  -s_K\left[ p_K A_{1,L}(K)+(1-p_K)B_{1,L}(K) \right]=0\ .
		\end{aligned}
	\end{equation}
	In the following we will show that the point $q=p$ is a global minimum of $\Pi_A-\Pi_Y.$ We start with the derivatives, remember $q_K=1-\sum_{i=1}^{K-1}q_i$:
	\begin{equation*}
		\begin{aligned}
			\frac{\partial (\Pi_A-\Pi_Y)}{\partial q_i}&=
			s_i\left[-B_{1,L}(i)+(p_i-2q_i)(A_{1,L}(i)-B_{1,L}(i)) \right]\\
			&-s_K\left[-B_{1,L}(K)+(p_K-2q_K)(A_{1,L}(K)-B_{1,L}(K)) \right]\ .
		\end{aligned}
	\end{equation*}
			For the derivative of the payoff difference where $p=q$ we get:
			\begin{equation*}\label{eq:derivative_p_eq_q}
				\begin{aligned}
					\frac{\partial (\Pi_A-\Pi_Y)}{\partial q_i}\bigg|_p					&=s_i\left[-B_{1,L}(i)-p_i(A_{1,L}(i)-B_{1,L}(i)) \right]\\
					&-s_K\left[-B_{1,L}(K)-p_K(A_{1,L}(K)-B_{1,L}(K)) \right]\ ,
				\end{aligned}
			\end{equation*}
			which is clearly equal to zero since equation~\eqref{eq:deriv_zero} is equal to zero.
			Next, we look at the second derivatives
			\begin{equation}
				\begin{aligned}
					\frac{\partial^2 (\Pi_A-\Pi_Y)}{\partial q_i q_i}
					=-2s_i(A_{1,L}(i)-B_{1,L}(i)) 
					-2s_K(A_{1,L}(K)-B_{1,L}(K))\ ,
				\end{aligned}\label{eq:hessian_qi}
			\end{equation}
			and for $i\neq j$
			\begin{equation}
				\begin{aligned}
					\frac{\partial^2 (\Pi_A-\Pi_Y)}{\partial q_i q_j}
					= -2s_K(A_{1,L}(K)-B_{1,L}(K)) \ .
				\end{aligned}\label{eq:hessian_qiqj}
			\end{equation}
			We know that all the elements of the Hessian matrix~\eqref{eq:hessian_qi} and~\eqref{eq:hessian_qiqj} are positive, since $A_{1,L}(i)-B_{1,L}(i)<0$ for all $i \in \{1,2,\dots, K\}$ and also all $s_i$ must be positive (since we consider the equivalence case for all $i$).
			Thus, the Hessian is symmetric with equal positive off-diagonal entries and diagonal entries $H_{ii}> H_{jk}$ for all $i,j,k \in \{1,2,\dots, K-1\}.$ It is easy to prove that such matrix is positive definite~(see Lemma~\ref{lemma:positive_definite}). The second derivative is also independent of $q$ and thus we see that the function  $\Pi_A-\Pi_Y$  is strictly convex in $q.$ This means that the minimum at $q=p$ is also a unique global minimum. From this it follows that 
			\[\Pi_A=\Pi(A;A^{L-2},Y) >  \Pi(Y;A^{L-2},Y)=\Pi_Y\ .\]
			Thus, we have shown that strategy $A$ is a unique ESS in the equivalence to zero case.

			Next, let us assume that not all derivatives are equal to zero, but there are some $i \in I,$  where $I \subseteq \{1,2,\dots,K-1\}$ for which the derivative is negative and therefore $p_i=0.$ Now, there are two groups of strategies $Y$ we have to look at. In the first group there are strategies $\{q_i\}_{i=1}^K$ such that $q_i=0, $ for all $i \in I.$ 
			In this case we can discard all the offers $s_i$ for $i \in I$ and consequently prove the evolutionary stability in the same way as in the previous part of the proof.
			
			In the second group, there are strategies for which $q_i\neq 0$ for some $i \in I.$ Then, due to the form of the derivatives we must have already for $j=0$ 
			\[\Pi(A;A^{L-1})>\Pi(Y;A^{L-1})\ .\] 
\end{proof}
\subsection*{Proposers' Subgame-Perfect Nash Strategy}\label{app:subs_symmetry}
In the following subsection, we prove an important result necessary for the analysis of the general case (see "$K$ Proposers and $L$ Responders" in the main manuscript), which states that if responders follow their evolutionarily stable strategy, the proposers' Nash equilibrium must be symmetric. Before proving the theorem, we introduce two lemmas that demonstrate properties of two functions essential for the proof of Theorem~\ref{thm:unique_prop_nash} and Proposition~\ref{prop:second_der}, where we prove that the second derivative of the payoff at the proposer's equilibrium is negative.
\begin{lema}\label{lemma:minus_gi_decreasing}
	Consider $L\geq 2$ and function 
	\begin{equation*}
		f: \mathbb{R} \to \mathbb{R}\ , \
		f_L(x)=\frac{1-(1-x)^{L}-Lx(1-x)^{L-1}}{x^2} \text{ and } f_L(0) = \frac{L(L-1)}{2}\ .
	\end{equation*}
	Then $f$ is continuous, positive and constant function for $L=2,$ and decreasing function for $L>2$ on interval $[0,1].$ 
\end{lema}
\begin{proof}
	Continuity is trivial to show. We will prove the other properties by mathematical induction.
	We start with $L=2$
	\begin{equation*}
		\begin{aligned}
			f_2(x)= \frac{1-(1-x)^{2}-2x(1-x)}{x^2}= 
			1\ ,
		\end{aligned}
	\end{equation*}
	$f_2$ is a constant function which is non-increasing.
	Now, we assume that for $f_N$ is non-increasing and look at $L=N+1,$ after some manipulations we arrive to:
	\begin{equation*}
		\begin{aligned}
			f_{N+1}(x)&
			&=f_N(x)+N(1-x)^{N-1}\ ,
		\end{aligned}
	\end{equation*}
	since $f_{N+1}$ is a sum of a non-increasing functions and a decreasing function, for $L>2,$ $f_L$ is decreasing. Since it is decreasing, we know that minimum on the interval $[0,1]$ is at $1$ where $f_L(1)=1.$ Thus $f_L$ is a positive function.
\end{proof}
\begin{lema}\label{lemma_ai_decreasing}
	Consider $L\geq 2$ and functions 
	\begin{equation*}
		f: \mathbb{R} \to \mathbb{R}\ , \
		f_L(x)= \frac{Lx^2(1-x)^{L-1}}{1-(1-x)^{L}-Lx(1-x)^{L-1}}\ , \ f_L(0) =\frac{2}{L-1}\ ,
	\end{equation*}
	Then $f_L$ is decreasing function on $(0,1]$ for $L>1.$ It is also positive on $[0,1)$ and $f(1)=0.$
\end{lema}
\begin{proof} 
	We will prove that the derivative of the function $f_L'(x)$ is negative on $(0,1]$ for any $L.$
	\[f_L'(x)=\frac{L x (1 - x)^{L-2} ((-L x + x - 2) (1 - x)^L - L x - x + 2)}{(1-(1-x)^{L}-Lx(1-x)^{L-1})^2}\ .\]
	We see it is sufficient to show that
	\[(-L x + x - 2) (1 - x)^L - L x - x + 2<0\ .\]
	We prove it by mathematical induction. For $L=2$ we have
	\[f_2'(x)=-x^3\ ,\]
	which is negative on $(0,1].$
	Now, we assume it holds for $L=N:$ 
	\[f_{N}'(x)=(-N x + x - 2) (1 - x)^N - N x - x + 2<0\ ,\]
	notice we have that
	\[(1-x)f_{N}'(x)=(-N x + x - 2) (1 - x)^{N+1}- N x - x + 2 +Nx^2+x^2-2x<0\ .\]
	When we look at $L=N+1$ we can rewrite $f'_{N+1}$ as:
	\begin{equation*}
		\begin{aligned}
			f_{N+1}'(x)&=(-(N+1) x + x - 2) (1 - x)^{N+1} - (N+1) x - x + 2
			= (1-x)f_N'(x)+x\left(1-(N+1)x-(1-x)^{N+1}\right)\ .
		\end{aligned}
	\end{equation*}
	Function $r(x)=1-(N+1)x-(1-x)^{N+1}$ is non-positive. At zero we have $r(0)=0$ and then the derivative is clearly negative for $x\in (0,1].$ 
	Thus, $f'_{N+1}(x)$ is a sum of two negative functions and therefore it is also negative. This means $f_L$ is a decreasing function on $(0,1].$ Additionally, $f_L(1)=0$ and therefore $f_L$ is positive on $[0,1).$
\end{proof}
After introducing the lemmas, we can proceed with the main theorem, which states that under the responders' evolutionarily stable strategy, the subgame-perfect Nash equilibrium must be symmetric.
\begin{theorem}\label{thm:unique_prop_nash}
Consider a MPMR UG with ${K\geq 2}$ proposers and $L \geq 2$ responders. If all responders follow the evolutionarily stable strategy described above and all offers are greater than zero, then the subgame-perfect Nash equilibrium of proposer strategies must be symmetric i.e.~ for all $i,j \in \{1,2,\dots,K\}:$ $s_i=s_j.$  
\end{theorem}
\begin{proof}
It is clear that if an offer leads to proposer's overall selection probability zero this can not constitute a Nash equilibrium, since by raising the offer to the highest offer (or in the case the highest offer is one then a sufficiently smaller offer than one) the proposer can reach a non-zero payoff. Thus, for Nash equilibria candidates we will only consider those sets of offers $\{s_1,s_2,\dots, s_K\}$ for which the evolutionarily stable strategy yields $p_i>0$ for all proposers $i\in\{1,2,\dots,K\}.$  If this is the case then also a feasible solution to the homogeneous system in~\eqref{app:system_s} exists. 

Let us remind that selection probabilities of proposers, $p_i(s_1,s_2,\dots,s_K)$ for $i \in \{1,2,\dots, K\},$ are functions of the offer set.
The payoff of the proposer $i$ is
\begin{equation}
    \begin{aligned}
        \Pi_{P,i}= (1-s_i)\left(1-(1-p_i)^L\right),
    \end{aligned}\label{eq:payoff_ith_proposer}
\end{equation}
and due to $\frac{\partial \Pi_R}{\partial p_i}=0$ we have that
\begin{equation}\label{eq:p_vs_C}
  \begin{aligned}
     \frac{s_i}{p_i}\left(1-(1-p_i)^L\right)=\frac{s_K}{p_K}\left(1-(1-p_K)^L\right)=:C(s_1,s_2,\dots s_K)\ .
         \end{aligned}
  \end{equation}
Thus, we can then rewrite the payoff in~\eqref{eq:payoff_ith_proposer} as
\[\Pi_{P,i}=1-(1-p_i)^L-Cp_i\ .\]
In Nash equilibrium we need the derivative of the proposer's payoff with respect to the strategy $s_i$ to be zero:  
\begin{equation}\label{dpayoffds}
    \begin{aligned}
        \frac{\partial \Pi_{P,i}}{\partial s_i}=L(1-p_i)^{L-1}\frac{\partial p_i}{\partial s_i}-C\frac{\partial p_i}{\partial s_i}-\frac{\partial C}{\partial s_i}p_i\ .
    \end{aligned}
\end{equation}
We denote $f(p_i)=\frac{1-(1-p_i)^L}{p_i}$ and have that $f(p_i)=\frac{C}{s_i},$ thanks to the inverse function theorem:
\begin{equation}\label{eq:dpds}
\begin{aligned}
\frac{\partial p_i}{\partial s_i}&=\frac{\partial f^{-1}}{\partial  \frac{C}{s_i}}\bigg|_{\frac{C}{s_i}}\frac{\partial\frac{C}{s_i}}{\partial s_i}=\frac{g_i}{s_i}\frac{\partial C}{\partial s_i}-g_i\frac{C}{s_i^2},\\
 \frac{\partial p_j}{\partial s_i} &=\frac{\partial f^{-1}}{\partial  \frac{C}{s_j}}\bigg|_{\frac{C}{s_i}}\frac{\partial\frac{C}{s_j}}{\partial s_i}=\frac{g_j}{s_j}\frac{\partial C}{\partial s_i} \quad  \text{for } i\neq j,
 \end{aligned}
\end{equation}
where $g_i=\frac{1}{f'(p_i)}=\frac{p_i^2}{(1-p_i)^{L-1}(Lp_i+1-p_i)-1}.$
Since we know that 
\begin{equation}\label{eq:dcds}
\begin{aligned}
    \sum_{j=1}^K \frac{\partial p_j}{\partial s_i}=0 \quad  \text{we have that} \quad 
\frac{\partial C}{\partial s_i} = \frac{1}{H}\frac{g_iC}{s_i^2}\ ,
\end{aligned}
\end{equation}

where $H=\sum_{j=1}^K \frac{g_j}{s_j}.$
By submitting~\eqref{eq:dpds} and~\eqref{eq:dcds} into~\eqref{dpayoffds} we get
\begin{equation}\label{eq:dpayoff_ds}
\begin{aligned}
\frac{\partial \Pi_{P,i}}{\partial s_i}=\frac{\partial p_i}{\partial s_i}\left(L(1-p_i)^{L-1}-C \right)-\frac{\partial C}{\partial s_i}p_i=\left(\frac{g_i^2C}{s_i^3H}-\frac{g_iC}{s_i^2}\right)\left(L(1-p_i)^{L-1}-C \right)-\frac{p_ig_iC}{s_i^2H}\ .
\end{aligned}
\end{equation}
In the subgame-perfect Nash equilibrium~\eqref{eq:dpayoff_ds} has to be equal to zero for all $i \in \{1,2,\dots,K\}.$ Thus, we can simplify set of equations in~\eqref{eq:dpayoff_ds} to
\begin{equation}\label{eq:system_deriv_PP}
\begin{aligned}
d_i:=\left(\frac{g_i}{s_i}-H\right)\left(L(1-p_i)^{L-1}-C \right)-p_i = 0\ , \text{  for all $i \in \{1,2,\dots,K\}\ .$}
\end{aligned}
\end{equation}

Next, we will show that if $p_i<p_j$ for some $i,j$ then $d_i>d_j$ which proves that other than the symmetric solution of a system of equations~\eqref{eq:system_deriv_PP} does not exist.
First we rewrite $\frac{g_i}{s_i}$ by using~\eqref{eq:p_vs_C} as
\[\frac{g_i}{s_i}=\frac{g_i  (1-(1-p_i)^L)}{Cp_i}\ ,\]
which can be further rewritten as
\begin{equation*}
    \begin{aligned}
        \frac{g_i}{s_i}&
        &=\frac{1}{C}\left(-p_i-\frac{Lp_i^2(1-p_i)^{L-1}}{1-(1-p_i)^{L}-Lp_i(1-p_i)^{L-1}}\right)\ ,
    \end{aligned}
\end{equation*}
from which we can derive that
\[ H=\sum_{i=1}^K \frac{g_i}{s_i} = \frac{1}{C}\left(-1 -\sum_{i=1}^K a_i \right),\]
where $a_i=\frac{Lp_i^2(1-p_i)^{L-1}}{1-(1-p_i)^{L}-Lp_i(1-p_i)^{L-1}}.$
Then we can rewrite $d_i$ (see~\eqref{eq:system_deriv_PP}) as
\begin{equation*}
    \begin{aligned}
       \left(1-p_i+\sum_{j=1,j\neq i}^K a_j\right)\left(L(1-p_i)^{L-1}-C \right)-Cp_i =0\ ,\\
    \end{aligned}
\end{equation*}
and derive
\[C(p_i)=\frac{(1-p_i+\sum_{j=1,j\neq i}^K a_j )L(1-p_i)^{L-1}}{1+\sum_{j=1,j\neq i}^K a_j }\ .\]
Now, we proceed to prove the statement by contradiction. Let us assume without loss of generality that $0<p_1<p_2< 1.$ Since $C(p_1)=C(p_2),$ we have
\begin{equation*}
    \begin{aligned}
          (1+A+a_1)(1-p_1+A+a_2)(1-p_1)^{L-1}&=(1+A+a_2)(1-p_2+A+a_1)(1-p_2)^{L-1},
    \end{aligned}
\end{equation*}
   where $A=\sum_{j=3}^K a_j,$ which means it should also be true that
\begin{equation*}
    \begin{aligned}
           (1+A+a_1)\left[(1-p_1+A)(1-p_1)^{L-1}+a_2(1-p_1)^{L-1}\right]=
           (1+A+a_2)\left[(1-p_2+A)(1-p_2)^{L-1}+a_1(1-p_2)^{L-1}\right],
    \end{aligned}
\end{equation*}
but as we show in the next steps, this is not possible. Since $p_1<p_2$ and function $a_i$ is decreasing and positive on $[0,1)$ (see Lemma~\ref{lemma_ai_decreasing}) we must have $a_1>a_2.$
That means that 
\[ (1+A+a_1)> (1+A+a_2) \quad  \text{ and } \quad (1-p_1+A)(1-p_1)^{L-1}> (1-p_2+A)(1-p_2)^{L-1}\ .\]
But as we will show also 
\[a_2(1-p_1)^{L-1}\geq a_1(1-p_2)^{L-1}\ ,\]
must hold. After submitting for $a_i,$ we have
\begin{equation}\label{eq:func_proof}
    \begin{aligned}
     \frac{Lp_2^2(1-p_2)^{L-1}}{1-(1-p_2)^{L}-Lp_2(1-p_2)^{L-1}}(1-p_1)^{L-1} &\geq   \frac{Lp_1^2(1-p_1)^{L-1}}{1-(1-p_1)^{L}-Lp_1(1-p_1)^{L-1}}(1-p_2)^{L-1}\\
     \frac{p_2^2}{1-(1-p_2)^{L}-Lp_2(1-p_2)^{L-1}}  &\geq  \frac{p_1^2}{1-(1-p_1)^{L}-Lp_1(1-p_1)^{L-1}}\ .\\  
    \end{aligned}
\end{equation}
We know the last line above is true thanks to Lemma~\ref{lemma:minus_gi_decreasing} which states that reciprocal of $\frac{p_2^2}{1-(1-p_2)^{L}-Lp_2(1-p_2)^{L-1}} $ is a positive and non-increasing function on $[0,1]$ and for $L\geq 2,$ meaning the function in the last line of~\eqref{eq:func_proof} is non-decreasing.
Thus, we showed that for any $p_1<p_2$ we can not have $C(p_1)=C(p_2).$  
\end{proof}
We conclude this section with a technical proposition showing that the second derivative of the proposers' payoff is negative in the equilibrium found in the main manuscript (see section named "$K$ Proposers and $L$ Responders"), thus proving, that it is indeed a maximum.
\begin{prop}\label{prop:second_der}
	Consider MPMR UG with $L\geq 2$ responders and $K\geq 2$ proposers. The second derivative of the proposer's payoff is negative at the equilibrium $s^*$ (see main manuscript "$K$ Proposers and $L$ Responders", subsection named "Proposers' Strategy").
\end{prop}
\begin{proof}
	Let us remind us that in the equilibrium $p=\frac{1}{K},$ $s^*>0$ and the derivative is
		\begin{equation*}
			\begin{aligned}
				\frac{\partial \Pi_{P,i}}{\partial s_i}
				&=\frac{g_iC}{s_i^2H}\left[\left(\frac{g_i}{s_i}-H\right)\left(L(1-p_i)^{L-1}-C \right)-p_i\right],
			\end{aligned}
		\end{equation*}
		where $H=\sum_{i=1}^K \frac{g_i}{s_i}$ and $g_i=\frac{p_i^2}{(1-p_i)^{L-1}(Lp_i+1-p_i)-1}$ is a function of $p_i$ that is negative and non-increasing (for the proof see Lemma~\ref{lemma:minus_gi_decreasing}). We know that $\frac{g_iC}{s_i^2H}>0$ thus, ${f_i=\left(\frac{g_i}{s_i}-H\right)\left(L(1-p_i)^{L-1}-C \right)-p_i}$ must be equal to zero in equilibrium $s^*.$ Then 
		\begin{equation*}
			\begin{aligned}
				\frac{\partial \Pi_{P,i}}{\partial s_i^2}\bigg|_{s=s^*}= \frac{\partial \frac{g_iC}{s_i^2H}}{\partial s_i}\bigg|_{s=s^*}\cdot 0 + \frac{g_iC}{s_e^2H} \frac{\partial f_i}{\partial s_i}\bigg|_{s=s^*},
			\end{aligned}
		\end{equation*}
		where
		\begin{equation*}
			\begin{aligned}
				\frac{\partial f_i}{\partial s_i}= -\sum_{j\neq i} \frac{g_j'}{s_j}\frac{\partial p_j}{\partial s_i}\left(L(1-p_i)^{L-1}-C \right)- \sum_{j\neq i} \frac{g_j}{s_j}\left( -L(L-1)(1-p_i)^{L-2}\frac{\partial p_i}{\partial s_i}- \frac{\partial C}{\partial s_i}\right)-\frac{\partial p_i}{\partial s_i}\ ,
			\end{aligned}
		\end{equation*}
		where $g_i' = \frac{\partial g_i}{\partial p_i}.$
		Now it is easy to see that $\frac{\partial f_i}{\partial s_i}\bigg|_{s=s^*}<0.$ The first term must be non-positive, since $g_j'$ is non-positive,  $\frac{\partial p_j}{\partial s_i}\bigg|_{s=s^*}< 0$ for $i\neq j$ and $(L(1-p_i)^{L-1}-C)>0$ in the equilibrium. The second and third term must be negative, since $g_j$ is negative and both $\frac{\partial p_i}{\partial s_i}$ and $\frac{\partial C}{\partial s_i}$ are positive in the equilibrium. Lastly, also $\frac{g_iC}{s_i^2H}>0,$ thus $\frac{\partial \Pi_{P,i}}{\partial s_i^2}\bigg|_{s=s^*}<0.$
	\end{proof}
\bibliography{references_SI}